\documentclass[12pt,reqno]{amsart}
\usepackage{graphicx, amsmath, amssymb, amsthm, amscd, fullpage,verbatim,hyperref,mathrsfs,latexsym}
\usepackage{pdfsync}
\theoremstyle{plain}
\newtheorem{theorem}{Theorem}[section]
\newtheorem{lemma}[theorem]{Lemma}
\newtheorem{proposition}[theorem]{Proposition}
\newtheorem{corollary}[theorem]{Corollary}
\newtheorem{claim}[theorem]{Claim}

\def\qed{\ifvmode\mbox{ }\else\unskip\fi\hskip 1em plus 10fill$\Box$}

\theoremstyle{definition}
\newtheorem{definition}{Definition}[section]

\newtheorem{remark}[definition]{Remark}

\newcommand{\e}{\varepsilon}

\newcommand{\R}{\mathbb{R}}
\newcommand{\conv}{\mathbf{\mathrm{conv}}}

\renewcommand{\le}{\leqslant}
\renewcommand{\ge}{\geqslant}
\renewcommand{\leq}{\leqslant}
\renewcommand{\geq}{\geqslant}
\renewcommand{\epsilon}{\varepsilon}
\renewcommand{\approx}{\asymp}
\newcommand{\eqdef}{\stackrel{\mathrm{def}}{=}}
\begin{document}
\title{Overlap properties of geometric expanders}
\author{Jacob Fox}
\author{Mikhail Gromov}
\author{Vincent Lafforgue}
\author{Assaf Naor}
\author{J\'anos Pach}

\date{}

\maketitle

\begin{abstract} The {\em overlap number} of a finite $(d+1)$-uniform hypergraph $H$ is defined as the largest constant $c(H)\in (0,1]$ such that no matter how we map the vertices of $H$ into $\R^d$, there is a point covered by at least a $c(H)$-fraction  of the simplices induced by the images of its hyperedges. In~\cite{Gro2}, motivated by the search for an analogue of the notion of graph expansion for higher dimensional simplicial complexes, it was asked whether or not there exists a sequence $\{H_n\}_{n=1}^\infty$ of arbitrarily large $(d+1)$-uniform hypergraphs with bounded degree, for which $\inf_{n\ge 1} c(H_n)>0$.
Using both random methods and explicit constructions, we answer this question positively by constructing infinite families of $(d+1)$-uniform hypergraphs with bounded degree such that their overlap numbers are bounded from below by a positive constant $c=c(d)$. We also show that, for every $d$, the best value of the constant $c=c(d)$ that can be achieved by such a construction is asymptotically equal to the limit of the overlap numbers of the complete $(d+1)$-uniform hypergraphs with $n$ vertices, as $n\rightarrow\infty$.
For the proof of the latter statement, we establish the following geometric partitioning result of independent interest. For any $d$ and any $\epsilon>0$, there exists $K=K(\epsilon,d)\ge d+1$ satisfying the following condition. For any $k\ge K$, for any point $q \in \mathbb{R}^d$ and for any finite Borel measure $\mu$ on $\mathbb{R}^d$ with respect to which every hyperplane has measure $0$, there is a partition $\mathbb{R}^d=A_1 \cup \ldots \cup A_{k}$ into $k$ measurable parts of equal measure such that
all but at most an $\epsilon$-fraction of the $(d+1)$-tuples $A_{i_1},\ldots,A_{i_{d+1}}$ have the property that either all simplices with one vertex in each
$A_{i_j}$ contain $q$ or none of these simplices contain $q$.
\end{abstract}



\tableofcontents

\section{Introduction}

Let $G=(V,E)$ be an $n$-vertex  graph. Think of $G$ as a $1$-dimensional simplicial complex, i.e., each edge is present in $G$ as an actual interval. Assume that for every subset $S\subseteq V$ of size $\left\lfloor \frac{n}{2}\right\rfloor$ the number of edges joining $S$ and $V\setminus S$ is at least $\alpha|E|$, for $\alpha\in (0,1]$. It follows that for every $f:V\to \R$, if we extend $f$ to be a linear (or even just continuous) function defined also on the edges of $G$, there must necessarily exist a point $x\in \R$ such that $|f^{-1}(x)|\ge \alpha |E|$. Indeed, $x$ can be chosen to be a median of the set $f(V)\subseteq \R$. In other words, no matter how we draw $G$ on the line, its edges will heavily overlap.


As  illustrated by this simple example, the above expander-like condition\footnote{It isn't quite edge expansion since we do not care about boundaries of small sets.} on $G$  implies that all of its embeddings in $\R$ satisfy a geometric overlap condition. This condition naturally extends to higher-dimensional simplicial complexes, and can thus serve as a potential definition of a higher-dimensional analogue of edge expansion\footnote{To be precise, what we are detecting here is only that $G$ contains a large expander, rather than being an expander itself.}. Such investigations of high-dimensional geometric analogues of edge expansion were initiated by Gromov in~\cite{Gro2}. The present paper follows this approach.

In 1984, answering a question of K\'arteszi, two undergraduates at E\"otv\"os University, Boros and F\"uredi \cite{BF}, proved the following theorem.
\begin{theorem}[\cite{BF}]\label{BorosFuredi} For every set
$P$ of $n$ points in the plane, there is a point (not necessarily in
$P$) that belongs to at least $\left(\frac{2}{9}-o(1)\right){n\choose
3}$ closed triangles induced by the elements of $P$.
\end{theorem}
The factor $\frac29$ in Theorem~\ref{BorosFuredi} is asymptotically tight, as shown by Bukh, Matou\v{s}ek and Nivasch in~\cite{BN}. A short and elegant ``book proof" of Theorem~\ref{BorosFuredi} was given by Bukh~\cite{Bu}. In Section~\ref{top}, we present an alternative ``topological" argument.

The theorem of Boros and F\"uredi has been generalized to higher dimensions. B\'ar\'any~\cite{Ba} proved that for every $d\in \mathbb N$ there exists a constant $c_d>0$ such that given any set $P$ of $n$ points in $\mathbb{R}^d$, one can always find a point in
at least $c_dn^d$ closed simplices whose vertices belong to $P$. In
fact, the following stronger statement due to Pach~\cite{Pa} holds true.

\begin{theorem}\label{pach} (\cite{Pa}) Every set $P$ of $n$ points in ${\mathbb{R}^d}$ has $d+1$ disjoint $\lfloor c'_dn\rfloor$-element subsets,
$P_1,\ldots,P_{d+1}$, such that {\em all} closed simplices with
one vertex from each $P_i$ have a point in common. Here $c'_d>0$
is a constant depending only on the dimension $d$.
\end{theorem}


Recall that a hypergraph $H=(V,E)$ consists of a set $V$ and a set $E$ of non-empty subsets of $V$. The elements of $V$ are called vertices and the elements of $E$ are called hyperedges. $H$ is {\em $d$-uniform} if every hyperedge $e\in E$ contains exactly $d$ vertices. The {\em degree} of a vertex $v\in V$ in $H$ is the number of hyperedges containing $v$. To simplify the presentation, we introduce the following terminology.

\begin{definition}\label{overlap} Given a $(d+1)$-uniform hypergraph $H=(V,E)$, its {\em overlap number} $c(H)$ is the largest constant $c\in (0,1]$ such that for every embedding $f:V\to \R^d$, there exists a point $p\in \R^d$ which belongs to at least $c|E|$ simplices whose vertex sets are hyperedges of $H$, i.e., there exists a set of hyperedges $S\subseteq E$ with $|S|\ge c|E|$ and $p\in \bigcap_{e\in S} \conv(f(e))$ (where $\conv(A)$ denotes the convex hull of $A\subseteq R^d$). An infinite family $\mathscr H$ of $(d+1)$-uniform hypergraphs is {\em highly overlapping} if there exists an absolute constant $c>0$ such that $c(H)>c$ for every $H\in \mathscr H$. An infinite family of {\em $d$-dimensional simplicial complexes} is called {\em highly overlapping} if the family of $(d+1)$-uniform hypergraphs consisting of the vertex sets of their $d$-dimensional faces (their $d$-skeletons) is highly overlapping\footnote{Gromov~\cite{Gro2} calls such simplicial complexes ``polyhedra with large cardinalities."}.
\end{definition}

Using this terminology, the Boros-F\"uredi theorem states that the family of all finite complete $3$-uniform hypergraphs (or $2$-skeletons of all complete simplicial complexes) is highly overlapping. B\'ar\'any's theorem says that the same is true for the family of complete $(d+1)$-uniform hypergraphs (or $d$-skeletons of complete simplicial complexes). The fact that the family of all finite complete graphs ($1$-skeletons of complete simplicial complexes) is highly overlapping (with $c=1/2$) is trivial, but its higher dimensional generalizations are much more subtle.

\medskip

It was a simple but very important graph-theoretic discovery by Pinsker~\cite{Pi} and others that there exist arbitrarily large edge {\em expanders} of bounded degree \cite{HLW}. As we have seen at the beginning of this paper, expanders with a fixed rate of expansion are necessarily highly overlapping. This fact motivated Gromov's question~\cite{Gro2} whether there exist infinite families of higher dimensional simplicial complexes with bounded degree that are highly overlapping. In other words, Gromov's question~\cite{Gro2} for $2$-dimensional simplicial complexes asks whether a Boros-F\"uredi type theorem remains true if instead of {\em all} triangles determined by $n$ points in the plane, we consider only ``sparse" systems of triangles. In particular, do there exist arbitrarily large $3$-uniform hypergraphs $H$, in which every vertex belongs to at most a constant number $k$ of triples, and whose overlap numbers are bounded from below by an absolute positive constant?

In Section~\ref{2/9}, we answer this question in the affirmative, by proving the following result.

\begin{theorem}\label{29result}
For any $\e>0$, there exists a positive integer $k=k(\e)$ satisfying
the following condition. There is an infinite sequence of
$3$-uniform hypergraphs $H_n$ with $n$ vertices and $n$ tending to infinity, each of degree $k$,
such that, for any embedding of the vertex set $V(H_n)$ in
$\mathbb{R}^2$, there is a point belonging to at least a
$(\frac{2}{9}-\e)$-fraction of all closed triangles induced by
images of hyperedges of $H_n$. Here the constant $\frac{2}{9}$
cannot be improved.
\end{theorem}

We also generalize Theorem~\ref{29result} to $(d+1)$-uniform hypergraphs with $d\geq 2$.

\begin{theorem}\label{answer}
For every integer $d\ge 2$, there exist positive constants $c_d$ and $k_d$ with the following property. There is an infinite sequence of $(d+1)$-uniform hypergraphs $H_n$ with $n$ vertices and $n$ tending to infinity, each of degree $k_d$, such that, for any embedding of the vertex set $V(H_n)$ in $\mathbb{R}^d$, there is a point in $\mathbb{R}^d$ that belongs to at least a $c_d$-fraction of all closed simplices induced by images of hyperedges of $H_n$.
\end{theorem}

Among the most natural and powerful methods to construct good expanders is the use of certain Cayley graphs of finitely generated groups (see~\cite{LPS88,Mar88,DSV03}), via arguments related to Kazhdan's property (T) (see~\cite{BHV08}). Such graphs yield explicit constructions of expanders that have extremal spectral properties, namely Ramanujan graphs~\cite{LPS88}. Being Cayley graphs of finitely generated groups, these constructions can be viewed as quotients of trees (Cayley graphs of free groups). It is natural to study hypergraph versions of this type of construction, based on quotients buildings (a type of higher dimensional simplicial complexes that extends the notion of a tree~\cite{steger}). In particular, a notion of Ramanujan complex, which is a simplicial complex with extremal spectral properties analogous to Ramanujan graphs, was introduced and constructed in~\cite{Bal00,CSZ03,Li04,LSV05,LSV05-exp,Sar07}. Here we show that such constructions can yield highly overlapping bounded degree hypergraph families. Specifically, we show that for every integer $r\ge 2$, for a large enough odd prime power $q$, certain finite quotients of the building of $PGL_r(F)$, where $F$ is a non-archimedian local field with residue field of order $q$, are highly overlapping $r$-uniform hypergraphs (with degree and overlap number depending only $q,r$). Rather than defining the relevant notions in the introduction, we refer to Section~\ref{sec:building} for precise definitions and statements. Instead, we state below the following concrete special case of our result, which follows from  our argument in Section~\ref{sec:building}, in combination with a construction of Lubotzky, Samuels and Vishne~\cite{LSV05-exp}.

\begin{theorem}\label{thm:PGL} For every odd prime $p$ and every integer $r\ge 3$ there exist $k(p,r)\in \mathbb N$ and $c(p,r)>0$ with the following property. For every $m\in \mathbb N$, the finite group $G=PGL_r(\mathbb F_{p^m})$, where $\mathbb F_{p^m}$ is the field of cardinality $p^m$, has a symmetric generating set $S\subseteq G$ of size bounded above by $k(p,r)$, such that the following holds. Consider the $r$-regular hypergraph $H$ whose vertex set is $G$ and whose hyperedges are those $r$-tuples $\{g_1,\ldots,g_r\}\subseteq G$ with $g_ig_j^{-1}\in S$ for all distinct $1\le i,j\le r$ (i.e., $H$ is the hypergraph consisting of all cliques of size $r$ in the Cayley graph induced by $S$). Then  there exist arbitrarily large integers $m$ for which the hypergraph $H$ has overlap number at least $c(p,r)>0$.
\end{theorem}


By Theorem~\ref{29result}, the best value of the constant $c_2$ in Theorem~\ref{answer} is close to $\frac29$, but in higher dimensions $d>2$, we do not have very good estimates for $c_d$. Our goal is to show, roughly speaking, that the best constant in Theorem~\ref{answer} is the same as the best constant in the Boros-F\"uredi-B\'ar\'any theorem (Theorem~\ref{BorosFuredi}). To state this formally, it will be convenient to introduce some notation. Let $c(K^{d+1}_n)$ be the overlap number of $K^{d+1}_n$, the complete $(d+1)$-uniform hypergraph on $n$ vertices, and set
$$c(d)=\lim_{n \to \infty} c\left(K^{d+1}_n\right).$$
It is easy to show, via  a straightforward point duplication argument, that the limit defining $c(d)$ exists, and the Boros-F\"uredi-B\'ar\'any theorem shows that $c(d)>0$, for every $d$.

One might suspect that if $H$ is a $(d+1)$-uniform hypergraph without isolated vertices, then $c(H) \leq c(d)+o(1)$, where the $o(1)$ term goes to $0$ as the number of vertices of $H$ tends to infinity. This is not the case. Consider, for example, the $(d+1)$-hypergraph $H^{d+1}_n$ on $n$ vertices, whose hyperedges are those sets of size $d+1$ that contain the first $d$ vertices. In any general position embedding of the vertices of $H^{d+1}_n$ in $\mathbb{R}^d$, any segment joining a pair of points sufficiently close and on opposite sides of the face consisting of the first $d$ vertices stabs all the simplices induced by the images of hyperedges of $H^{d+1}_n$. Hence, $c(H^{d+1}_n) \geq 1/2$. However, $c(d)$ decays to $0$ at least exponentially in $d$ (see, e.g., \cite{Ba,BN}). Despite this example, we  show that our
suspicion is {\em correct} for bounded degree hypergraphs.

\begin{theorem}\label{tightness}
For any $d$, $\Delta\in \mathbb N$, and $\epsilon>0$, there is $n(d,\Delta,\epsilon)\in \mathbb N$ such that every $(d+1)$-uniform hypergraph $H$ on $n \geq n(d,\Delta,\epsilon)$ nonisolated vertices with maximum degree $\Delta$ satisfies $c(H) \leq c(d)+\epsilon$.
\end{theorem}

In the other direction, we show that there are regular $(d+1)$-uniform hypergraphs $H$ of bounded degree such that $c(H)$ is at least $c(d)-\epsilon$ for any given $\epsilon>0$.

\begin{theorem}\label{usestructure}
For each $d\in \mathbb N$ and $\epsilon>0$, there is $r(d,\epsilon)\in \mathbb N$ such that for every $r \geq r(d,\epsilon)$ and sufficiently large $n$ which is a multiple of $d+1$, there is a $(d+1)$-uniform, $r$-regular hypergraph $H$ on $n$ vertices with $c(H) \geq c(d)-\epsilon$.
\end{theorem}

\noindent The previous two theorems essentially show that $c(d)$ is the largest possible overlap number for bounded degree hypergraphs with sufficiently many nonisolated vertices.

\medskip

The proof of the last theorem is based on a geometric partitioning result of independent interest. A $(d+1)$-tuple of subsets $S_1,\ldots,S_{d+1} \subseteq \mathbb{R}^d$ is said to be {\em homogeneous} with respect to a point $q \in \mathbb{R}^d$ if either all simplices with one vertex in each of the sets $S_{1},\ldots,S_{d+1}$ contain $q$, or none of these simplices contain $q$.

\begin{theorem}\label{mainstructure}
For a positive integer $d$ and $\epsilon>0$, there exists another positive integer $K=K(\epsilon,d)\ge d+1$ such that for any $k \geq K$ the following statement is true. For any point $q \in \mathbb{R}^d$ and for any finite Borel measure $\mu$ on $\mathbb{R}^d$ with respect to which  every hyperplane has measure $0$, there is a partition $\mathbb{R}^d=A_1 \cup \ldots \cup A_{k}$ into $k$ measurable parts of equal measure such that all but at most an $\epsilon$-fraction of the $(d+1)$-tuples $A_{i_1},\ldots,A_{i_{d+1}}$ are homogenous with respect to $q$. \end{theorem}

An {\it equipartition} of a finite set is a partition of the set into subsets whose sizes differ by at most one. A discrete version of Theorem~\ref{mainstructure} is the following.

\begin{corollary}\label{discretestructure}
Given a positive integer $d$ and $\epsilon>0$, there exists another positive integer $K=K(\epsilon,d)\ge d+1$ such that for any $k \geq K$ the following statement is true. For any finite set $P \subseteq \mathbb{R}^d$ and for any point $q \in \mathbb{R}^d$, there is an equipartition $P=P_1 \cup \ldots \cup P_k$ such that all but at most an $\epsilon$-fraction of the $(d+1)$-tuples $P_{i_1},\ldots,P_{i_{d+1}}$ are homogenous with respect to $q$. \end{corollary}

Notice that due to B\'ar\'any's result~\cite{Ba} that $c(d)>0$, by taking $\epsilon \ll c(d)$, Corollary \ref{discretestructure} immediately implies Theorem \ref{pach}.

\medskip

The rest of the paper is organized as follows. Section~\ref{top} contains a detailed topological proof of the Boros-F\"uredi theorem (Theorem~\ref{BorosFuredi}), following  Gromov's approach in~\cite{Gro2}. In the two subsections of Section~\ref{random}, we present randomized constructions for Theorems~\ref{29result} and~\ref{answer}. In the plane, these constructions are nearly optimal; their overlap numbers are close to the value $\frac29$.  In Section~\ref{sec:det}, we give a deterministic recipe how to turn certain families of explicitly given expander graphs into families of highly overlapping $(d+1)$-uniform hypergraphs. In Section~\ref{sec:building} we give a criterion which ensures that certain finite quotients of the building of $PGL_r(F)$ are highly overlapping $r$-uniform hypergraphs; this criterion implies in particular Theorem~\ref{thm:PGL}. In Section~\ref{szemeredi}, we establish a Szemer\'edi-type theorem for infinite hypergraphs with a measure on their vertex sets (Theorem~\ref{thm:regularity}). This is used in Section~\ref{structuresub} for the proof of the geometric partition result Theorem~\ref{mainstructure}. In Section~\ref{optimal}, we show how this result can be applied to obtain Theorem~\ref{usestructure}. Section~\ref{tightsubsection} contains the proof of Theorem~\ref{tightness}.

For the sake of clarity of the presentation, in the rest of this paper, we systematically omit floor and ceiling signs whenever they are not crucial. We shall also assume throughout that all embeddings of hypergraphs into $\R^d$ are such that the vertices are mapped to points in general position. Even though the corresponding statements for degenerate embeddings will then follow from standard limiting arguments, it is convenient to make this assumption in order to not deal explicitly with such degeneracies in each of the proofs.

\section{A topological proof of the Boros-F\"uredi theorem}\label{top}

We will prove a somewhat stronger statement. Given a set $P$ of $n$ points in the plane, a {\em ray} (closed half-line) is said to be {\em exposed} if it has nonempty intersection with fewer than $n^2/9$ segments connecting point pairs in $P$. The set of all segments connecting two elements of $P$ forms a {\em complete geometric graph} $K(P)$ on the vertex set $P$, and we refer to these segments as the {\em edges} of $K(P)$.

\begin{proposition}\label{rays}
Given a set $P$ of $n$ points in the plane, one can always find a point $q$ not necessarily in $P$ such that no ray emanating from $q$ is exposed.
\end{proposition}

Suppose that such a point $q$ does not belong to $P$.
For each $p\in P$, ray emanating from $q$ in the direction opposite to $p$ intersects at least $n^2/9$ edges of $K(P)$. Each such edge, together with $p$,
spans a triangle that contains $q$. Every triangle is counted at most {\em three} times, therefore the total number of triangles containing $q$ is at least $n(n^2/9)/3=n^3/27$. If $q$ belongs to $P$, the number of (closed) triangles containing $q$ is larger than $n^3/27$.

Thus, it is sufficient to prove Proposition~\ref{rays}. Suppose for a contradiction that for each point $q$ of the plane, there is an exposed ray emanating from $q$. Let $D$ denote a large disk around the origin $O$, which contains all elements of $P$, and let $S^1$ denote the boundary of $D$. For $\sigma\in \R^2\setminus\{O\}$, we denote by ${\rm ray}(q,\sigma)$ the ray emanating from $q$ in the direction parallel to $\overrightarrow{O\sigma}$.

Notice that for any two exposed rays, ${\rm ray}(q,\sigma)$ and ${\rm ray}(q,\tau)$, emanating from the same point, one of the two closed regions bounded by them contains fewer than $n/3$ points of $P$. Otherwise, one of the regions has $x$ points of $P$ with $n/3 \leq x \leq 2n/3$, and the two
boundary rays together would intersect at least $x(n-x) \geq (n/3)(2n/3)=2n^2/9$ edges, which implies at least one of them was not
exposed.

Let $I$ denote the set of all pairs $(q,\varrho)\in D\times S^1$, for which ${\rm ray}(q,\varrho)$ is exposed or belongs to the closed region bounded by two exposed rays, ${\rm ray}(q,\sigma)$ and ${\rm ray}(q,\tau)$, that contains fewer than $n/3$ points of $P$.

\begin{claim}\label{properties}The set $I$ has the following properties:
\begin{enumerate}
\item[(a)] $I$ is an {\em open} subset of $D\times S^1$,
\item[(b)] $(\varrho,\varrho)\in I$ for all $\varrho\in S^1$,
\item[(c)] for every $q\in D$, the set $I_q\stackrel{\mathrm{def}}{=}\{\varrho\in S^1:\  (q,\varrho)\in I\}$ is a nonempty proper subinterval of $S^1$.
\end{enumerate}
\end{claim}

\begin{proof}
Parts (a) and (b) directly follow from the definition. It is also clear, by our contrapositive assumption, that $I_q$ is a nonempty interval for every $q\in D$.

We have to show only that $I_q\neq S^1$. To see this, let ${\rm ray}(q,\varrho)$ be an exposed ray
emanating from $q$, and let $\varrho'\in S^1$ be a direction such that both closed regions bounded by ${\rm ray}(q,\varrho)$ and ${\rm ray}(q,\varrho')$ contain at least $n/2$ points of $P$.

We claim that $\varrho'\not\in I_q$. Otherwise, we can select two exposed rays, ${\rm ray}(q,\sigma)$
and ${\rm ray}(q,\tau)$, such that ${\rm ray}(q,\varrho')$ belongs to the closed region bounded by
them which contains fewer than $n/3$ points. The {\em three} rays, ${\rm ray}(q,\varrho)$,
${\rm ray}(q,\sigma)$, and ${\rm ray}(q,\tau)$, cut the plane into
{\em three} closed regions, and it is easy to see that each of them must contain fewer than
$n/3$ points, which is a contradiction. Indeed, if e.g. the region bounded by ${\rm ray}(q,\varrho)$ and
 ${\rm ray}(q,\sigma)$ that does not contain ${\rm ray}(q,\tau)$ had at least $n/3$ points, then by the discussion above
the closure of its complement had fewer than $n/3$ points, contradicting our assumption that both closed regions bounded
${\rm ray}(q,\varrho)$ and ${\rm ray}(q,\varrho')$ contain at least $n/2$ points.
\end{proof}

Now we can obtain the desired contradiction, thus completing the proof of  Proposition~\ref{rays},
by applying to $J\stackrel{\mathrm{def}}{=}(D\times S^1)\setminus I$ the following version of the Brouwer fixed point theorem.

\begin{lemma}\label{lem:brouwer}
Let $J$ be a closed subset of $D\times S^1$ with the property that for every $q\in D$ we have that $J_q\stackrel{\mathrm{def}}{=}\{\varrho\in S^1:\  (q,\varrho)\in J\}$ is a nonempty proper (closed) subinterval of $S^1$. Then $(\varrho,\varrho)\in J$, for some $\varrho\in S^1$.
\end{lemma}
To see why Lemma~\ref{lem:brouwer} holds true, assume for contradiction that $(\varrho,\varrho)\notin J$, for all $\varrho\in S^1$. Write $J_S\stackrel{\mathrm{def}}{=}J\cap (S^1\times S^1)$, and let $Proj_1, Proj_2:J_S\to S^1$ denote the projections onto the first and second coordinates, respectively. The fibers of $Proj_1$ are nonempty proper closed intervals, and therefore $Proj_1$ induces a bijection between $\pi_1(J_S)$ and $\pi_1(S^1)=\mathbb Z$. But, the contrapositive assumption implies that $Proj_1$ and $Proj_2$ are homotopic, and therefore $Proj_2$ also induces a bijection between $\pi_1(J_S)$ and $\pi_1(S^1)$. This is a contradiction since $Proj_2$ extends to $J$, and $\pi_1(J)=0$ since $J$ is fibered over $D$ with fibers equal to intervals.

Clearly, Lemma~\ref{lem:brouwer} contradicts part (b) of Claim~\ref{properties}. \qed

\section{Sparse constructions using the probabilistic method}\label{random}

In this section, we prove Theorems~\ref{29result} and~\ref{answer} using the probabilistic method. Our planar construction is nearly optimal, but in higher dimensions the overlap numbers of our hypergraphs will be far from maximal. We note that our proofs use a non-uniformly random choice of $(d+1)$-uniform hypergraphs of degree $k_d$, which is designed especially for our purposes. Nevertheless, the argument in Section~\ref{optimal}, which uses Theorem~\ref{mainstructure}, shows that assuming the degree $r$ satisfies a large enough lower bound depending on $d$ (which is inferior to the bound on $k_d$ obtained in this section), for a hypergraph $H$ chosen uniformly at random among all  $(d+1)$-uniform hypergraphs of degree $r$, with high probability $c(H)$ will be bounded below by a positive constant depending only on $d$ (which is also inferior to the bound on $c_d$ obtained in this section).

\subsection{Highly overlapping triple systems---Proof of Theorem~\ref{29result}}\label{2/9}

The outline of the proof of Theorem \ref{29result} is the following. We first pick $t$ randomly and independently selected partitions of the set $[n]=\{1,2,\ldots,n\}$ into parts of equal size $b$. We define $H_n$ to be the 3-uniform hypergraph with vertex set $[n]$, consisting of all triples that lie in the same part in at least one of the $t$ partitions. Finally, we will show that $H_n$ meets the requirements of Theorem \ref{29result}.

We need the following simple technical lemma. A key ingredient that is used in the proof is the Chernoff bound for negatively associated random variables (see, e.g., \cite{DD}). It implies that if $A_1,\ldots,A_n$ are $n$ mutually negatively correlated events in an arbitrary probability space such that $A_i$ has probability $p_i$, then the probability that the number of $A_i$ which occur exceeds the expected number $p_1+\cdots+p_n$ by at least $a$ is at most $e^{-2a^2/n}$.

\begin{lemma}\label{lemmafor29}
Suppose that $\delta>0$, and let $b=\delta^{-3}$, $\beta=2e^{-2\delta^2b}$, $r=4\beta^{-2}b$, $t=r\delta^{-1}$. If $n$ is a sufficiently large multiple of $b$, then there exist $t$ partitions $\mathcal{P}_1,\ldots,\mathcal{P}_t$ of
$[n]$, each consisting of $n/b$ parts of size $b$, with the following two properties:
\begin{enumerate}
\item any two parts of size $b$ in different partitions have at most {\em two} elements in common,
\item for every subset $S \subseteq [n]$, there are fewer than $r$ partitions $\mathcal{P}_i$ for which at least $\beta n/b$ parts contain at least $\left(\frac{|S|}{n}+\delta \right)b$ elements of $S$.
\end{enumerate}
\end{lemma}

\begin{proof}
We verify that $t$ {\em randomly} selected partitions of $[n]$ into parts of equal size $b$ almost surely have the desired properties. Fix a set $S \subseteq [n]$, and consider a random partition $\mathcal{P}$ of $[n]$ into parts $I_1,\ldots,I_{n/b}$ of size $b$.
For any $1\le i\le n/b$, let $A_i$ denote the event that $|I_i \cap S| \geq \left(\frac{|S|}{n}+\delta\right)b$. For any $1 \leq j \leq b$, let $A_{i,j}$ denote the event that the $j$th element of $I_i$ is in $S$. The events $A_{i,1},\ldots,A_{i,b}$ are mutually negatively correlated and each of them has probability $|S|/n$. Thus, by Chernoff's bound~\cite{DD}, we have
$$
\Pr[A_i] \leq e^{-2(\delta b)^2/b} = e^{-2\delta^2b}=\frac{\beta}{2}.
$$
Let $X$ denote the event that at least $\beta n/b$ of the events $A_1,\ldots,A_{n/b}$ occur. Since the events $A_1,\ldots,A_{n/b}$ are also mutually negatively correlated and each has probability at most $\beta/2$, we can again apply the Chernoff bound~\cite{DD} to obtain $$\Pr[X] \leq e^{-2(\frac{\beta n}{2b})^2/(n/b)}=e^{-\frac{1}{2}\beta^2 n/b}.$$

Take $t$ independent random partitions of $[n]$,  $\mathcal{P}_1,\ldots,\mathcal{P}_t$, each consisting of $n/b$ parts of size $b$.
The probability that a given pair of parts of size $b$ have at least $3$ elements in common is at most ${b \choose 3}\left(\frac{b}{n}\right)^3 \leq \frac{b^6}{6n^3}$. Since there are ${tn/b \choose 2}$ such pairs, by linearity of expectation, the probability that there is a pair sharing at least $3$ elements is at most ${tn/b \choose 2}\frac{b^6}{6n^3}<\frac{t^2b^4}{12n}$. Hence, by our choice of parameters, almost surely condition (1) will be satisfied.

For a fixed $S\subseteq [n]$, the probability that for at least $r$ of the partitions $\mathcal{P}_1,\ldots,\mathcal{P}_t$, at least $\beta n/b$ of the $b$-element subsets of the partition have at least $\left(\frac{|S|}{n}+\delta\right)b$ elements in $S$ is at most $${t \choose r}(\Pr[X])^r \leq {t \choose r}e^{-r\frac{1}{2}\beta^2 n/b}={t \choose r}e^{-2n} \leq e^{-n}.$$ The number of subsets $S$ of $[n]$ is $2^n$. Hence, by linearity of expectation,
the expected number of subsets $S$ with property (2) is $o(1)$. We conclude that there are $t$ such partitions with the desired properties.
\end{proof}

Let $\delta=\e/50$ and $k=t{b-1 \choose 2}$. Consider  the
$3$-uniform hypergraph $H_n$ with $V(H_n)=[n]$, the hyperedges of
which are those triples that lie in the same part in at least one
(hence, precisely one) of the partitions
$\mathcal{P}_1,\ldots,\mathcal{P}_t$ meeting the requirements of
Lemma \ref{lemmafor29}. Clearly, in $H_n$, each vertex belongs to
$k=t{b-1 \choose 2}$ hyperedges.

The proof of Theorem~\ref{29result} can now be completed by adapting the idea of Bukh \cite{Bu}.
Consider an embedding of the vertices of $H_n$ in the plane. We shall use the following lemma of Ceder~\cite{Ce}:

\begin{lemma}[Ceder~\cite{Ce}]\label{lem:ceder}
Assume that $n$ is divisible by $6$. Given any set of $n$ points in the plane, there are three concurrent lines that divide the plane
into $6$ angular regions, each containing roughly the same number of points. More precisely, there are
disjoint $\frac{n}{6}$-element point sets $S_1,\ldots,S_6$ such that $S_i$ is contained in the closure
of region $i$.
\end{lemma}
We shall assume throughout the $n$ is divisible by $6$. Let $S_1,\ldots,S_6$ be the sets from Lemma~\ref{lem:ceder}, and let $p$ denote the intersection point of the {\em three} lines from Lemma~\ref{lem:ceder}. By a simple case analysis,
Bukh~\cite{Bu} showed that, for every choice of {\em six} points, one from each $S_i$, at least $8$ of the
${6 \choose 3}=20$ triangles induced by them contain $p$.


Let $I \subseteq [n]$ be a $b$-element set such that $|I \cap S_i| \leq (\frac{|S_i|}{n}+\delta)b=(1+6\delta)\frac{b}{6}$, for $1 \leq i \leq 6$. Obviously, we have
$$|I \cap S_i| \geq b-5(1+6\delta)\frac{b}{6} \ge (1-30\delta)\frac{b}{6},$$
for every $i$. Each of the $$\prod_{i=1}^6 |I \cap S_i| \geq (1-30\delta)^6\left(\frac{b}{6}\right)^6 $$
$6$-element sets with one vertex from each $I \cap S_i$ induces at least $8$ triangles that contain point $p$. Each of these triangles belongs to at most
$(1+6\delta)^3(\frac{b}{6})^3$ such $6$-element sets. Thus, there are at least $$8\frac{(1-30\delta)^6\left(\frac{b}{6}\right)^6}{(1+6\delta)^3(\frac{b}{6})^3} \geq
\frac{1}{27}(1-200\delta)b^3 > (1-200\delta)\frac{2}{9}{b \choose 3}$$ triangles induced by {\em three} vertices in $I$ which contain $p$.

According to part 2 of Lemma~\ref{lemmafor29}, for every $i,\;$ $1 \leq i \leq 6$, fewer than $r$ partitions $\mathcal P_j$ have the property that at least $\beta \frac{n}{b}$ of their parts contain at least
$(1+6\delta)\frac{b}{6}$ elements of $S_i$. Hence, the total number of $b$-element parts $I$ in all $t$ partitions, for which
 $|I \cap S_i| > (1+6\delta)\frac{b}{6}$ for some $i,\;$ $1 \leq i \leq 6$, is smaller
than
$$6r\frac{n}{b}+6t\beta \frac{n}{b} = 6\delta t\frac{n}{b}+6\beta t \frac{n}{b} \leq 10 \delta t\frac{n}{b}.$$ It follows that the fraction of the $t\frac{n}{b}{b \choose 3}$ hyperedges of $H_n$ that contain point $p$ in this embedding is at least $$(1-10\delta)(1-200\delta)\frac{2}{9}\geq  (1-210\delta)\frac{2}{9} \geq \frac{2}{9}-\e,$$ which completes the proof of Theorem \ref{29result}. \qed

\subsection{Higher dimensions---Proof of Theorem~\ref{answer}}\label{sec:answer}

As in the proof of Theorem~\ref{29result}, we establish Theorem~\ref{answer} using Lemma \ref{lemmafor29}. We may assume that
$c'_{d}=1/m$ with $m$ an integer, where $c_d'$ is the constant in Theorem~\ref{pach}, and let
$n$ be a multiple of $m$. Set $\delta=\frac{1}{2m(m-1)}$ and apply Lemma \ref{lemmafor29}. Consider now the $(d+1)$-uniform hypergraph $H_n$ with $V(H_n)=[n]$, the hyperedges of which are those $(d+1)$-element sets that lie in the same part in at least one (hence, precisely one) of the partitions $\mathcal{P}_1,\ldots,\mathcal{P}_t$ meeting the requirements of Lemma \ref{lemmafor29}. Clearly, in $H_n$, each vertex belongs to $k_d=t{b-1 \choose d}$ hyperedges.

Consider now any embedding of $V(H_n)$ into $\mathbb{R}^d$, and let
$P$ denote the image of $V(H_n)$. By Theorem \ref{pach}, one can
find disjoint $c'_d n$-element subsets $P_1,P_2,\ldots,P_{d+1}\subseteq P$ and a point $q$ such that picking one element from each subset $P_i$, their convex hull always contains $q$. We extend this to a partition $P=P_1 \cup \ldots \cup P_{m}$ into subsets of size $n/m$ by picking the $P_i$ for $d+1<i \leq m$ of size $n/m$ arbitrarily.

Let $I \subseteq [n]$ be a $b$-element set such that $$\forall1 \leq i \leq m,\quad |I \cap P_i| \leq \left(\frac{|P_i|}{n}+\delta\right)b=\left(1+\frac{1}{2(m-1)}\right)\frac{b}{m}.$$
 Obviously, we have
$$|I \cap P_i| \geq b-(m-1)\left(1+\frac{1}{2(m-1)}\right)\frac{b}{m} = \frac{b}{2m},$$
for every $1 \leq i \leq m$. Each of the $$\prod_{i=1}^{d+1} |I \cap P_i| \geq \left(\frac{b}{2m}\right)^{d+1}$$
$(d+1)$-element sets with one vertex from each $I \cap P_1,\ldots,I\cap P_{d+1}$ induces a closed simplex containing point $q$. Hence, the fraction of
$(d+1)$-element subsets of $I$ which induce a closed simplex that contains point $q$ is at least $$\left(\frac{b}{2m}\right)^{d+1}{b \choose d+1}^{-1} \geq (d+1)!\left(\frac{c_d'}{2}\right)^{d+1}.$$

According to part (2) of Lemma~\ref{lemmafor29}, for every  $1 \leq i \leq m$, fewer than $r$ partitions $\mathcal P_j$ have the property that at least $\beta \frac{n}{b}$ of their parts contain at least
$(\frac{|P_i|}{n}+\delta)b$ elements of $P_i$. Hence, the total number of $b$-element parts $I$ in all $t$ partitions, for which
 $|I \cap P_i| > (\frac{|P_i|}{n}+\delta)b$ for some  $1 \leq i \leq m$, is smaller than
$$mr\frac{n}{b}+mt\beta \frac{n}{b} = m\delta t\frac{n}{b}+m\beta t \frac{n}{b} \leq \frac{3}{4}\cdot t\cdot \frac{n}{b}.$$ Hence, the fraction of the $t\frac{n}{b}{b \choose d+1}$ hyperedges of $H_n$ that contain the point $q$ in this embedding is at least $\frac{1}{4}(d+1)!\left(\frac{c_d'}{2}\right)^{d+1}$.\qed

\section{Deterministic constructions using expander graphs}\label{sec:det}

In the next two subsections, we present deterministic constructions based on expander graphs, to provide alternative proofs of Theorem~\ref{29result} and Theorem~\ref{answer}. These proofs yield significantly better bounds on $k(\e)$ and $k_d$ in Theorem~\ref{29result} and Theorem~\ref{answer}, respectively. As in the previous section, the proof gives a nearly optimal bound in the plane, but not in higher dimension.

\subsection{Highly overlapping triple
systems---second proof of Theorem~\ref{29result}}\label{sec:sharp}

Fix integers $k,n\in \mathbb N$, with $n$ divisible by $6$, and  let $G=(\{1,\ldots,n\},E)$ be a
$k$-regular graph on the vertex set $\{1,\ldots,n\}$. Let
$k=\lambda_1\ge \lambda_2\ge \cdots \ge\lambda_n$ be the eigenvalues
of the adjacency matrix of $G$ in decreasing order, and write $\lambda=\max_{i\in
\{2,\ldots,n\}} |\lambda_i|$. For any $S,T\subseteq \{1,\ldots, n\}$
let $E(S,T)$ denote the number of ordered pairs $(i,j)\in S\times T$
such that $ij\in E$. The expander mixing lemma (see Corollary 9.2.5
in~\cite{AlSp}) states that
\begin{equation}\label{eq:mix}
\left|E(S,T)-\frac{k|S|\cdot |T|}{n}\right|\le \lambda\sqrt{|S|\cdot
|T|}.
\end{equation}

For every $i\in \{1,\ldots,n\}$ let $N_G(i)\stackrel{\mathrm{def}}{=}\{j\in \{1,\ldots,n\}:\ ij\in E\}$ denote its neighborhood in $G$.  Define a hypergraph $H$ on the vertex set $\{1,\ldots,n\}$ by letting  $E(H)$ consist of those triples
$\{i,j,\ell\}$ for which there exists $r\in \{1,\ldots,n\}$ such that $ir,jr,\ell r\in E$, i.e., $i,j,\ell\in N_G(r)$.  Assume from now on that the graph $G$ is quadrilateral-free. This implies that the hyperedges in $H$ corresponding to three vertices $i,j,\ell\in N_G(r)$ cannot arise from neighborhoods of vertices of $G$ other than $r$ itself. Hence the $3$-uniform hypergraph corresponding to $H$ is $k\binom{k-1}{2}$-regular and $|E(H)|= \binom{k}{3}n$.

Fix $\e,\delta\in (0,1)$.  Let $\{P_i\}_{i=1}^6$ be
a partition of $\{1,\ldots,n\}$ such that $|P_j|= \frac{n}{6}$ for
all $1\le j\le 6$. Write
$$
A_j=\left\{i\in \{1,\ldots,n\}:\ |N_G(i)\cap
P_j|<\frac{(1-\delta)k}{6}\right\}.
$$
Then, by definition, we have $E(A_j,P_j)< |A_j|\frac{(1-\delta)k}{6}$. An application
of~\eqref{eq:mix} yields the inequality:
\begin{equation*}\label{eq:A_j}
|A_j|\frac{(1-\delta)k}{6}\ge \frac{k|A_j|\cdot |P_j|}{n}-\lambda
\sqrt{|A_j|\cdot |P_j|}=\frac{k|A_j|}{6}-\lambda
\sqrt{\frac{n|A_j|}{6}},
\end{equation*}
which simplifies to
$$
|A_j|\le \frac{6\lambda^2n}{\delta^2k^2}.
$$
Thus, if we define
\begin{equation}\label{eq:def A}
A=\left\{i\in \{1,\ldots,n\}:\ |N_G(i)\cap
P_j|\ge\frac{(1-\delta)k}{6}\ \forall j\in \{1,\ldots,6\}\right\},
\end{equation}
then
\begin{equation}\label{eq:size A}A\ge n-\sum_{j=1}^6 |A_j|\ge
n\left(1-\frac{36\lambda^2}{\delta^2k^2}\right).
\end{equation}
We shall assume from now on that $\frac{36\lambda^2}{\delta^2k^2}<1$. We also note that for every $i\in A$ and $j\in \{1,\ldots, 6\}$ we have
\begin{equation}\label{eq:upper bound deduce}
|N_G(i)\cap
P_j|\le k-\sum_{r\in \{1,\ldots,6\}\setminus \{j\}}|N_G(i)\cap
P_r|\le k-5\frac{(1-\delta)k}{6}=\frac{(1+5\delta)k}{6}.
\end{equation}

Let $x_1,\ldots,x_n\in \R^2$ be an embedding of $\{1,\ldots,n\}$ in the plane. Let $S_1,\ldots,S_6$ be a partition of $\{x_1,\ldots, x_n\}$, as in the first proof of Theorem~\ref{29result}, which corresponds to the three concurrent lines from Lemma~\ref{lem:ceder}, whose common intersection point is $p\in \R^2$. We shall use the above reasoning (and notation) for the partition $P_1,\ldots,P_6$ of $\{1,\ldots,n\}$ given by $P_j=\{i\in \{1,\ldots,n\}:\ x_i\in S_i\}$.

Fix $i\in A$, where $A$ is as in~\eqref{eq:def A}. For every $(j_1,\ldots,j_6)\in \prod_{r=1}^6 (N_G(i)\cap P_{r})$ at least $8$ of the 20 triangles induced by the points $\{x_{j_1},\ldots,x_{j_6}\}$ contain $p$. By the definition of $A$, there are at least $\left(\frac{(1-\delta)k}{6}\right)^6$ such $6$-tuples, while, using~\eqref{eq:upper bound deduce}, each of these triangles that contains $p$ belongs to at most $\left(\frac{(1+5\delta)k}{6}\right)^3$ such $6$-tuples. Observe also that by the definition of $H$, since all of these triangles correspond to neighbors of $i$, their corresponding triples of indices belong to $E(H)$, and since $G$ is quadrilateral-free, they cannot arise from the above reasoning with $i$ replaced by any other vertex. Thus, the number of triangles that are images of hyperedges of $H$ and contain $p$ is at least
\begin{equation}\label{eq:before choices}
8\cdot \frac{\left(\frac{(1-\delta)k}{6}\right)^6}{\left(\frac{(1+5\delta)k}{6}\right)^3}\cdot |A|\stackrel{\eqref{eq:size A}}{\ge} \frac{(1-\delta)^6k^3}{27(1+5\delta)^3}n\left(1-\frac{36\lambda^2}{\delta^2k^2}\right)=
\left(1-O\left(\delta+\frac{\lambda^2}{\delta^2k^2}+\frac{1}{k}\right)\right)\cdot \frac29 \binom{k}{3}n.
\end{equation}

For arbitrarily large  $n$, we can choose the graph $G$ so that it
is quadrilateral-free and $\lambda\le 2\sqrt{k}$ (e.g., Ramanujan
graphs work---see~\cite{LPS88,HLW}). By choosing $\delta \approx \e$
and $k \approx \frac{1}{\e^{3}}$ in~\eqref{eq:before choices}, we
get that $p$ is in at least $\left(\frac{2}{9}-\e\right)|E(H)|$ of
the triangles in that are images of hyperedges of $H$. Note that the
degree of $H$ is $O(k^3)=O\left(\frac{1}{\e^9}\right)$. This proves
Theorem~\ref{29result} with the bound
$k(\e)=O\left(\frac{1}{\e^9}\right)$.\qed

\subsection{Higher dimensions---second proof of Theorem~\ref{answer}}\label{sec:high dim}
Here we shall use a variant of the construction in
Section~\ref{sec:sharp}, to give an alternative proof of
Theorem~\ref{answer}. We use the notation from
Section~\ref{sec:sharp}, and we assume that $k\ge d$. Fix $n$
vectors $x_1,\ldots,x_n\in \R^d$. Define a set of $d$-dimensional
simplices $H'$ whose vertices are in $\{x_1,\ldots,x_n\}$ by taking
the simplex whose vertices are the distinct vectors
$\{x_{j_1},x_{j_2},\ldots,x_{j_{d+1}}\}$ if and only if we have
$j_1j_2,j_2j_3,\ldots j_dj_{d+1}\in E$. In other words, the
simplices in $H'$ correspond to non-returning walks of length $d$ in
$G$. Thus, $|H'|\le k^{d}n$.

Let $P_1,\ldots, P_{d+1}\subseteq \{x_1,\ldots,x_n\}$ be the disjoint subsets from Theorem~\ref{pach}, i.e., $|P_i|\ge c_d'n$, and all the closed simplices with one vertex in each of the sets $\{P_1,\ldots,P_{d+1}\}$ have a point in common. Set $Q_i=\{j\in \{1,\ldots,n\}:\ x_j\in P_i\}$. Define  $\widetilde Q_{d+1}=Q_{d+1}$ and inductively for $i\in \{2,\ldots,d+1\}$,
$$
\widetilde Q_{i-1}=\left\{j\in  Q_{i-1}:\ \exists \ell\in \widetilde Q_i\ j\ell\in E\right\}.
$$
Then, by definition, there are no edges between $Q_{i-1}\setminus \widetilde Q_{i-1}$ and $\widetilde Q_i$. It follows from~\eqref{eq:mix} that
$$
\frac{k}{n}\left|Q_{i-1}\setminus \widetilde Q_{i-1}\right|\cdot \left|\widetilde Q_i\right|\le
\lambda \sqrt{\left|Q_{i-1}\setminus \widetilde Q_{i-1}\right|\cdot \left|\widetilde Q_i\right|}.
$$
Thus, we have
$$
\frac{\lambda^2n^2}{k^2}\ge \left(\left|Q_{i-1}\right|-\left| \widetilde Q_{i-1}\right|\right) \left|\widetilde Q_i\right|\ge  \left(c_d'n-\left| \widetilde Q_{i-1}\right|\right) \left|\widetilde Q_i\right|,
$$
or
\begin{equation}\label{eq:recursion}
\left| \widetilde Q_{i-1}\right|\ge c_d'n- \frac{\lambda^2n^2}{k^2\left|\widetilde Q_i\right|}.
\end{equation}

Assuming that $\lambda\le \frac{c_d'}{2}k$, inequality~\eqref{eq:recursion} implies by induction that for all $i\in \{1,\ldots,d+1\}$ we have $\left| \widetilde Q_{i-1}\right|\ge \frac{c_d'}{2}n$ (for $i=d+1$ this follows from our assumption, arising from Theorem~\ref{pach}, on the cardinality of $P_{d+1}$). Thus, $\left| \widetilde Q_1\right|\ge \frac{c_d'}{2}n$, and by construction any point $j\in \widetilde Q_1$ can be completed to a walk in $G$ of length $d$ whose $i$th vertex is in $Q_i$. Each such walk corresponds to a simplex in $H'$, and by Theorem~\ref{pach}, all of these simplices have a common point. Thus, the number of simplices in $H'$ which have a common point is at least $\frac{c_d'}{2}n\ge \frac{c_d'}{2k^{d}}|H'|$. Since there exist arbitrarily large graphs $G$ with $\lambda\le \frac{c_d'}{2}k$ and $k\le k_d$ (e.g., for Ramanujan graphs we can take $k_d\approx \frac{1}{(c_d')^2}$), this completes our deterministic proof of Theorem~\ref{answer}.\qed

\section{Finite quotients of buildings}\label{sec:building}

Let $F$ be a non-archimedean local  field, $\mathcal O_{F}$ its ring of integers, $\pi_{F}$ a uniformizer and  $q=|\mathcal O_{F}/\pi_{F}\mathcal O_{F}|$  the cardinality of the residue field of $F$. For example,  we may take $F=\mathbb Q_{p}$,
$\mathcal O_{F}=\mathbb Z_{p}$, $\pi_{F}=p$ and $q=p$. We assume below that $q$ is odd.

Let $r\geq 3$ be an integer and $G=PGL_{r}(F)$. Now $r$ is fixed but $F$ will be chosen such that $q$ is big enough. We recall that $K= PGL_{r}(\mathcal O_{F})$  is a maximal compact subgroup of $G$.
We also recall that $G/K$ is the set of vertices of a building, and is also equal to the set of  lattices in $F^{r}$ up to homothety (a lattice in $F^{r}$ is a free $\mathcal O_{F}$-submodule of rank $r$ and a homothety is the multiplication by an element of the multiplicative group $F^{\times}$). We refer to~\cite{steger} for an elementary introduction to the building of $PGL_{r}(F)$ (we will not use the definition of a building below---all simplicial complexes will be defined explicitly).
 We have  the map
$$\mathrm{type} : G/K\to \mathbb Z/r \mathbb Z$$ such that if $x\in G/K$ is the homothety class of a  lattice $M\subseteq F^{r}$ and  $\mathrm{det}(M)=\pi_{F}^{a}\mathcal O_{F}^{\times}$ with $a\in \mathbb Z$, then $\mathrm{type}(x)=a\text{  mod  }r\mathbb Z$.
We denote by  $\mathrm{vdet}(\cdot)$ the composition
$$G\xrightarrow{\mathrm{det}} F^{\times}/(F^{\times})^r\xrightarrow{\mathrm{valuation}}\mathbb Z/r \mathbb Z,$$
and for $i\in  \mathbb Z/r \mathbb Z$ we write $G_{i}=\mathrm{vdet}^{-1}(\{i\})$. Thus, $G_{0}$ is a subgroup of index $r$ in $G$ and the $G_{i}$ are the left and right cosets for $G_{0}$ in $G$. We remark  that $K\subseteq G_{0}$.
 For $g\in G_{i}$ and $x\in G/K$ we have $ \mathrm{type}(gx)=i+\mathrm{type}(x)$. Moreover, $G_{i}/K$ is the subset of $G/K$ of vertices of type $i$.

Let $\Lambda=\{(\lambda_{1},...,\lambda_{r})\in \mathbb  Z^{r}/\mathbb Z(1,...,1):\  \lambda_{1}\leq ... \leq \lambda_{r}\}$ be the set of dominant coweights of $PGL_{r}(F)$.
For $(\lambda_{1},...,\lambda_{r})\in \Lambda$ we write $$D(\lambda_{1},...,\lambda_{r})=
\begin{pmatrix} \pi_{F}^{\lambda_{1}} & 0 & \dots& \dots&0 \\
  0 & \pi_{F}^{\lambda_{2}}& \ddots& \ddots & \vdots\\
  \vdots & \ddots & \pi_{F}^{\lambda_{3}}  & \ddots & \vdots\\
            \vdots & \ddots & \ddots& \ddots &0\\
              0 & \dots & \dots &0&\pi_{F}^{\lambda_{r}}
                       \end{pmatrix}\in G.$$
Then the mapping $\delta\stackrel{\mathrm{def}}{=} (\lambda_{1},...,\lambda_{r})\mapsto KD(\lambda_{1},...,\lambda_{r}) K$ is a bijection from $\Lambda$ to $K\backslash G/K$. For $x,y\in G/K$ we have $x^{-1}y\in K\backslash G/K$, and we define the relative position of $x$ and $y$ in the building as $\sigma(x,y)=\delta^{-1}(x^{-1}y)\in \Lambda$. In other words, for $(\lambda_{1},...,\lambda_{r})\in \Lambda$ we have $\sigma(x,y)=(\lambda_{1},...,\lambda_{r})$ if and only if there exists a basis $(e_{1},...,e_{r})$ of $F^{r}$ such that $x$ is the homothety class of
$\mathcal O_{F}e_{1}+...+\mathcal O_{F}e_{r}$
and $y$ is the homothety class of
$\pi_{F}^{\lambda_{1}}\mathcal O_{F}e_{1}+...+\pi_{F}^{\lambda_{r}}\mathcal O_{F}e_{r}$. The inverse image of $K\backslash G_{0}/K$  under $\delta $ is
$$\Lambda_0=\left\{(\lambda_{1},...,\lambda_{r})\in \mathbb Z^{r}
/\mathbb Z(1,...,1)
: \lambda_{1}\leq ... \leq \lambda_{r}\ \wedge\  \sum_{i=1}^{r}\lambda_{i}=0 \text{ in } \mathbb Z/r\mathbb Z\right\}\subseteq \Lambda,$$
which is also the set of dominant coweights of $SL_{r}(F)$.

We recall that for $d\in \{0,...,r-1\}$ a $d$-dimensional face of the building is a $(d+1)$-tuple $\{y_{0},...,y_{d}\}\subseteq G/K$ such that there are lattices $M_{0},...,M_{d}$ in $F^{r}$ satisfying
\begin{enumerate}
\item
$M_{0}\subsetneq M_{1} \subsetneq ... \subsetneq  M_{d} \subsetneq  \pi_{F}^{-1}M_{0},$
\item
 $y_{0},...,y_{d}$ are the homothety classes of $M_{0},...,M_{d}$.
 \end{enumerate}

The type of a face $\{y_{0},...,y_{d}\}$ is $\{\mathrm{type}(y_{0}),...,\mathrm{type}(y_{d})\}\subseteq \mathbb Z/r \mathbb Z$. For any non-empty  subset $I\subseteq \mathbb Z/r \mathbb Z$ we denote by $Y_{I}$ the set of faces of type $I$ in the building. More precisely, we denote by $Y_{I}$ the set of families $(y_{i})_{i\in I}$ with $y_{i}\in G_{i}/K$ such that $\{y_{i}:\  i\in I\}$ is a $(|I|-1)$-dimensional face of the building.

Let $G^{+}\subseteq G$ be the
subgroup generated by unipotent elements of $G$. Then $G^{+}$ is also the image of $SL_{r}(F)$ in $G=PGL_{r}(F)$, and $G^{+}\subseteq G_{0}$. We have an exact sequence $$1\to G^{+}\to G \xrightarrow{\mathrm{det}} F^{\times}/(F^{\times})^r\to 1.$$
  We will apply  Theorem 1.1 of~\cite{hee-oh} (a quantitative form of Kazhdan's property (T))  to the strongly orthogonal system
 of the set of roots of $G=PGL_{r}(F)$ consisting of the single root $e_{1}-e_{r}$ (where $e_{i}-e_{j}$ stands for the root corresponding to the character of the maximal torus of diagonal matrices which sends
 $D(\lambda_{1},...,\lambda_{r})\in G$ to $\pi_{F}^{\lambda_{i}-\lambda_{j}}$, for
 $(\lambda_{1},...,\lambda_{r})\in \Lambda$).
 By this theorem,   for any unitary representation $(H,\pi)$ of $G$ without a  nonzero $G^{+}$-invariant vector, for any     $(\lambda_{1},...,\lambda_{r})\in \Lambda$
 and for any $K$-invariant vectors $\xi,\eta\in H$ we have
 \begin{gather}\label{est-hee-oh0}
 \big|\langle \eta, \pi(D(\lambda_{1},...,\lambda_{r}))\xi\rangle\big|\leq \frac{(\lambda_{r}-\lambda_{1})(q-1)+(q+1)}{q^{\frac{1}{2}(\lambda_{r}-\lambda_{1})}(q+1)}
 \|\xi\|\cdot \|\eta\|
.\end{gather}
  For any nonzero
 $(\lambda_{1},...,\lambda_{r})\in \Lambda_{0}$ we have $\lambda_{r}\geq \lambda_{1}+2$ and therefore  inequality \eqref{est-hee-oh0}
 implies
 \begin{gather}\label{est-hee-oh}
 \big|\langle \eta, \pi(D(\lambda_{1},...,\lambda_{r}))\xi\rangle\big| \leq \frac{3}{q}\|\xi\|\cdot \|\eta\| \text{ \ for any nonzero \ }
 (\lambda_{1},...,\lambda_{r})\in \Lambda_{0}.\end{gather}
 Note that in  \eqref{est-hee-oh0} we did not use  the full strength of Theorem 1.1 of~\cite{hee-oh} because we used the very poor strongly orthogonal system
 $\{e_{1}-e_{r}\}$, whereas~\cite{hee-oh} uses the maximal  strongly orthogonal system
 $\left\{e_{1}-e_{r}, e_{2}-e_{r-1},...,e_{\lfloor \frac{r}{2}\rfloor}-e_{r+1-\lfloor \frac{r}{2}\rfloor}\right\}$ to get optimal bounds. However, the optimal bounds do not give in~\eqref{est-hee-oh} an exponent of $q$  better than $-1$  for $
 (\lambda_{1},...,\lambda_{r})=(-1,0,...,0,1)$.

Let $\Gamma$ be a  cocompact lattice in $G_{0}$ satisfying the condition
\begin{itemize}
\item[(C)] for any $x\in G/K$ and any $\gamma\in \Gamma$ different from $1$, the distance from $x$ to $\gamma x$ along the $1$-skeleton of $G/K$ is $>2$.
\end{itemize}

The quotient  $X=\Gamma\backslash G/K$ is the set of vertices of a  simplicial complex
whose faces are the quotient by $\Gamma $ of the faces of the building $G/K$. Since  $\Gamma\subseteq G_{0}$, we have an obvious  type function $\mathrm{type}:X\to \mathbb Z/r \mathbb Z$,  and thanks to the condition (C), $X$ has no multiple edges.
  For any non-empty subset $I\subseteq \mathbb Z/r \mathbb Z$ we  write $X_{I}=\Gamma\backslash Y_{I}$ and we identify $X_{I}$ with the set of faces of type $I$ in $X$.
We note that for $i\in \mathbb Z/r \mathbb Z$, $X_{\{i
\}}\stackrel{\mathrm{def}}{=}\Gamma\backslash  G_{i}/K$ is the subset of vertices of type $i$ in $X$ and is of cardinality $\frac{1}{r}|X|$.


  For  any   $(\lambda_{1},...,\lambda_{r})\in \Lambda$ we have a normalized Hecke operator $H_{\lambda_{1},...,\lambda_{r}}$ acting on $\ell^{2}(\Gamma\backslash G/K)$ in the following way: we view elements of $\ell^{2}(\Gamma\backslash G/K)$ as $\Gamma$-invariant functions on $G/K$ and
$H_{\lambda_{1},...,\lambda_{r}}(f)(x)$ is the average of the values of $f$ on the  vertices $y\in G/K$  which satisfy $\sigma(x,y)=(\lambda_{1},...,\lambda_{r})$.
For $\xi,\eta\in \ell^{2}(\Gamma\backslash G/K)$ considered as $K$-invariants vectors in $L^{2}(\Gamma\backslash G)$, endowed with the representation $\pi$ of $G$ by right translations,
we have \begin{gather}
\label{hecke-matrice}\langle \eta, H_{\lambda_{1},...,\lambda_{r}} \xi \rangle =
\langle \eta, \pi(D(\lambda_{1},...,\lambda_{r}))\xi\rangle.\end{gather}

The subspace  of $G^{+}$-invariant vectors in $L^{2}(\Gamma\backslash G)$ is $L^{2}(\Gamma\backslash G)^{G^{+}}=L^{2}(F^{\times}/(F^{\times})^r\det(\Gamma))$ because $G/G^{+}=F^{\times}/(F^{\times})^r$ is abelian. Let $H$ be the orthogonal complement of $L^{2}(\Gamma\backslash G)^{G^{+}}$ in $L^{2}(\Gamma\backslash G)$. The representation $(H,\pi)$ does not have any nonzero $G^{+}$-invariant vector and we will apply inequality \eqref{est-hee-oh} to it. For any set $Z$ endowed with  a finite measure  we write
 $L^{2}_{0}(Z)$ for the hyperplane of $L^{2}(Z)$ which is orthogonal to the constant function $1$.
For $(\lambda_{1},...,\lambda_{r})\in \Lambda_{0}$, $H_{\lambda_{1},...,\lambda_{r}}$ acts diagonally on the direct sum decomposition $\ell^{2}(\Gamma\backslash G/K)=\bigoplus_{i\in \mathbb Z/r\mathbb Z}\ell^{2}(\Gamma\backslash G_{i}/K)$ and
the estimate
\eqref{est-hee-oh}, together with \eqref{hecke-matrice},
implies  the operator norm bound
 \begin{gather}\label{est-hee-oh2}
 \|H_{\lambda_{1},...,\lambda_{r}}\|_{\mathcal L\left(\ell^{2}_{0}(\Gamma\backslash G_{i}/K)\right)}
 \leq \frac{3}{q}, \end{gather}
 for any $i\in \mathbb Z/r \mathbb Z$ and
for any  nonzero $(\lambda_{1},...,\lambda_{r})\in \Lambda_{0}$.
To justify \eqref{est-hee-oh2} it remains to check that for any
 $i\in \mathbb Z/r \mathbb Z$ and any $f\in \ell^{2}_{0}\left(\Gamma\backslash G_{i}/K\right)$, the extension of $f$ by $0$ to $\Gamma\backslash G/K$, considered as an element of $L^{2}(\Gamma\backslash G)$, belongs to $H$, i.e., is orthogonal to $L^{2}(\Gamma\backslash G)^{G^{+}}$. Indeed  $G^{+}K=G_{0}$, since we have $\det(K)=\mathcal O_{F}^{\times}/(\mathcal O_{F}^{\times})^{r}$ and  an exact sequence
 $$1\to \mathcal O_{F}^{\times}/(\mathcal O_{F}^{\times})^{r}\to F^{\times}/(F^{\times})^{r}\to \mathbb Z/r \mathbb Z\to 0.$$

\begin{proposition}\label{prop-alpha}  Let $I\subseteq \mathbb Z/r \mathbb Z$ be a non-empty set. For $i\in I$, fix $\alpha_{i}>0$ and a subset $Z_{i}\subseteq X_{\{i\}}$ of cardinality $\geq \alpha_{i} \left|X_{\{i\}}\right|$. Then  the proportion of elements of $(x_{i})_{i\in I}\in
 X_{I}$ satisfying $x_{i} \in Z_{i}$ for every $i\in I$  is at least
$$
\prod_{i\in I}\alpha_{i}-\frac{2(|I|-1)}{\sqrt{q}}.
$$
\end{proposition}

\begin{proof}  Proposition~\ref{prop-alpha} is obvious for $|I|=1$, and  it follows by induction on $|I|$ due to the following lemma.
 \end{proof}


\begin{lemma}\label{lem-alpha-beta}  Fix $I\subseteq \mathbb Z/r \mathbb Z$ of cardinality $\geq 2$, and $i\in I$. Write $I'=I\setminus \{i\}$. For $\alpha,\alpha'>0$,  $Z\subseteq X_{\{i\}}$ of cardinality $\geq \alpha \left|X_{\{i\}}\right|$ and
$Z'\subseteq X_{I'}$ of cardinality $\geq \alpha' \left|X_{I'}\right|$,  the proportion of $(x_{j})_{j\in I}\in
 X_{I}$ satisfying   $x_{i}\in Z$  and $(x_{j})_{j\in I'}\in
 Z'$  is at least
 $$
 \alpha\alpha'-\frac{2}{\sqrt{q}}.
 $$
\end{lemma}

\begin{proof} Let $T:\ell^{2}(X_{\{i\}})\to \ell^{2}(X_{I'})$ be the following averaging operator: for any $f\in \ell^{2}\left(X_{\{i\}}\right)$ and $(x_{j})_{j\in I'}\in
X_{I'}$, $T(f)((x_{j})_{j\in I'})$ is the average of $f(x_{i})$ over $x_{i}\in X_{\{i\}}$  such that  $(x_{j})_{j\in I}$ belongs to $X_{I}$.
We normalize the norms of $\ell^{2}(X_{\{i\}})$ and $\ell^{2}(X_{I'})$ such that
the constant function $1$ has norm $1$.
We denote by $T_{0}: \ell^{2}_{0}(X_{\{i\}})\to \ell^{2}_{0}(X_{I'})$ the restriction of $T$ to the hyperplane orthogonal to the constant function $1$.

Note that $T^{*}T: \ell^{2}(X_{\{i\}})\to \ell^{2}(X_{\{i\}})$ is also an averaging operator, and in fact it is an average of the Hecke operators $H_{\lambda_{1},...,\lambda_{r}}$ for  $(\lambda_{1},...,\lambda_{r})\in \Lambda_{0}$ such that $\lambda_{r}\leq\lambda_{1}+ 2$.
In this average, the coefficient of $H_{0,...,0}=\mathrm{Id}$ is $\leq q^{-1}$.
Indeed, let $\widetilde i\in \mathbb Z$ be a lifting of $i$ and let $\widetilde i_{-}, \widetilde i_{+}$ be the biggest integer $<\widetilde i$ (resp. the smallest integer $>\widetilde i$) whose images $i_{-},i_{+}$ belong to $I'$. Then
 for any $(x_{j})_{j\in I'}\in
X_{I'}$, the number of $x_{i}\in X_{\{i\}}$ such that $(x_{j})_{j\in I}$ belongs to $X_{I}$ is exactly $\left|\mathrm{Gr}\left(\widetilde i-\widetilde i_{-},\widetilde i_{+}-\widetilde i_{-} \right)(\mathbb F_{q})\right|$, the number of sub-$\mathbb F_{q}$-vector spaces of dimension $\widetilde i-\widetilde i_{-}$ in $\mathbb F_{q}^{\widetilde i_{+}-\widetilde i_{-}}$. The coefficient of $H_{0,...,0}$ is
$\left|\mathrm{Gr}\left(\widetilde i-\widetilde i_{-},\widetilde i_{+}-\widetilde i_{-} \right)(\mathbb F_{q})\right|^{-1}$, and it is clear that $\left|\mathrm{Gr}\left(\widetilde i-\widetilde i_{-},\widetilde i_{+}-\widetilde i_{-} \right)(\mathbb F_{q})\right|\geq q$.
By~\eqref{est-hee-oh2} we have  $$\left\|T_{0}^{*}T_{0}\right\|_{\mathcal{L}\left(\ell^{2}_{0}(X_{\{i\}})\right)}\leq \frac{1}{q}+\frac{3}{q}=\frac{4}{q}, $$ implying that $\|T_{0}\|_{\mathcal{L}\left(\ell^{2}_{0}(X_{\{i\}}),\ell^{2}_{0}(X_{I'})\right)}\leq  2 q^{-\frac{1}{2}} $. Therefore \begin{gather}\label{est-hee-oh3}\left|\left\langle T(\mathbf{1}_{Z}),\mathbf{1}_{Z'}\right\rangle- \frac{|Z|\cdot|Z'|}{\left|X_{\{i\}}\right|\cdot |X_{I'}|}\right|\leq 2 q^{-\frac{1}{2}}\sqrt{ \frac{|Z|\cdot|Z'|}{\left|X_{\{i\}}\right|\cdot |X_{I'}|}}\leq 2 q^{-\frac{1}{2}},\end{gather}  completing the proof of Lemma~\ref{lem-alpha-beta}.
\end{proof}

 \begin{remark}\label{rem:ramanujan}
 If $X=\Gamma\backslash G/K$ is assumed to be a Ramanujan complex (see~\cite{LSV05}) we can improve the estimate  in~\eqref{est-hee-oh2},  but at the end, in~\eqref{est-hee-oh3}, we get the same exponent $-\frac{1}{2}$.
\end{remark}

Proposition~\ref{prop-alpha} has the following corollary.

\begin{corollary}\label{cor-alpha}
Let $I\subseteq \mathbb Z/r \mathbb Z$ be of cardinality $\geq 2$.
Let $\mathcal X_{I}$ be the $|I|$-uniform hypergraph with vertices $X=\Gamma\backslash G/K$ and with hyperedges the set $X_{I}$ of faces of $X$ of type $I$. Then for $q$ large enough (as a function of $|I|$ alone) the overlap parameter of $\mathcal X_{I}$ is bounded
below by a positive constant depending only on $|I|$.
\end{corollary}
\begin{proof} Let $f:X\to \R^{|I|-1}$ be  an injection.
By the main result of~\cite{Pa} (which is a strengthening of  Theorem~\ref{pach}), applied to the sets $f\left(X_{\{i\}}\right)\subseteq\R^{|I|-1}$ for $i\in I$,
 we  obtain sets $P_i\subseteq f\left(X_{\{i\}}\right)$ with $|P_i|\ge c_{|I|-1}''\left|X_{\{i\}}\right|$, for $i\in I$,  such that all the simplices with vertices in each of the $P_i$ have a point in common.
 Here $c_{|I|-1}''$ is a constant depending only on $|I|$.
 By an application of Proposition~\ref{prop-alpha} to the sets $Z_i=f^{-1}(P_i)$ with $\alpha_{i}=c_{|I|-1}''$, we see that if $$q>\frac{16(|I|-1)^2}{\left(c_{|I|-1}''\right)^{2|I|}}$$ then the overlap parameter of $\Gamma\backslash G/K$ is at least $\frac12\left(c_{|I|-1}''\right)^{|I|}$.
\end{proof}

\begin{lemma}\label{lem-existence-lattices}  For any integer $r$ and for any non-archimedian local field $F$,   there exists a  cocompact lattice $\Gamma$ in $ G_{0}$ which satisfies condition (C),  and a sequence of finite index subgroups $\Gamma_{n}\subseteq \Gamma$ with  $\lim_{n\to\infty}| \Gamma/\Gamma_{n}|=\infty$.
\end{lemma}
\begin{proof} We start by choosing any cocompact arithmetic subgroup $\widetilde \Gamma$ in $G=PGL_{r}(F)$. The existence of a cocompact arithmetic subgroup is well-known: use division algebras over function fields when the characteristic of $F$ is finite and unitary groups over number fields when the characteristic of $F$ is $0$. In the case of characteristic $0$, it is a particular case of a theorem of Borel and Harder~\cite{BH78}; see also
Example 5.1.4, Corollary 5.12 and the remark after it in~\cite{Ben04}, for a short proof.
 Since $G/G_{0}=\mathbb Z/r\mathbb Z$, $\overline   \Gamma=\widetilde \Gamma\cap G_{0}$ is of finite index in $\widetilde \Gamma$. Since
 $\overline   \Gamma$ is a cocompact lattice in $G_{0}$,
 the  elements  $\gamma\in  \overline   \Gamma\setminus\{1\} $,
such that there exists $x\in G/K$ with $d(x,\gamma x)\leq 2$,
form a finite number  of conjugacy classes. By its arithmetic nature, $\overline   \Gamma$ embeds in the product of its finite quotients. Therefore there is a finite index subgroup $\Gamma\subseteq\overline   \Gamma$   which satisfies condition (C)   and a sequence of finite index subgroups  $\Gamma_{n}$  of  $\Gamma$ with  $\lim_{n\to\infty}| \Gamma/\Gamma_{n}|=\infty$.
\end{proof}

Corollary~\ref{cor-alpha}, applied (for $r\geq 3$) to the lattices $\Gamma_{n}$  of Lemma~\ref{lem-existence-lattices},  yields highly overlapping families of $|I|$-uniform hypergraphs.
In the  particular case where $I=\mathbb Z/r \mathbb Z$, these hypergraphs are finite quotients of the building of $PGL_{r}(F)$ and their  hyperedges are the images of the
 chambers of the building. 

\begin{proof}[Proof of Theorem~\ref{thm:PGL}] Theorem~\ref{thm:PGL} is a consequence of Corollary~\ref{cor-alpha}, applied to the Ramanujan complexes constructed in~\cite{LSV05-exp} (which are based on a lattice construction from~\cite{CS98}), together with their description in~\cite{LSV05-exp} as the clique complexes corresponding to the Cayley graphs associated to certain (explicitly defined) generators of $PGL_r(F)$, where $F=\mathbb F_p((t))$. We just need to ensure that the corresponding lattices are sublattices of $G_{0}$ (so that they preserve the type), as well as that condition (C) is satisfied. This is true for arbitrarily large $m$ due to Corollary 6.8 of~\cite{LSV05-exp}, or equivalently the case $r=d$ of Theorem 7.1 in~\cite{LSV05-exp}. Alternatively, one can consider the construction of the Ramanujan complexes in~\cite{Sar07}, specifically the second extreme distinguished case of Corollary 36 in~\cite{Sar07}.
 \end{proof}

\section{A Szemer\'edi-type theorem for infinite hypergraphs}\label{szemeredi}

In a measurable space, an {\it atom} is a measurable set which has positive measure and contains no set of smaller but positive measure. A measure which has no atoms is called {\it non-atomic}. A basic result of Sierpinski \cite{Si} states that if $\mu$ is a non-atomic measure, then for any measurable set $A$ and any $b$ with $0 \leq b \leq \mu(A)$, there is a measurable subset $B$ of $A$ with $\mu(B)=b$. Given a measurable space on a set $V$ and a measure $\mu$, the $h$-fold product measurable space is generated by all sets $B_1 \times \cdots \times B_h$, with $B_1,\ldots,B_h$ measurable subsets of $V$, and the product measure $\mu^h$ is the unique measure on this space given by $\mu^h(B_1 \times \cdots \times B_h)=\mu(B_1)\cdots\mu(B_h)$.

The aim of this section is to establish a Szemer\'edi-type theorem for infinite hypergraphs with a non-atomic measure on their vertex sets. The statement will be used in the next section, for the proof of Theorem~\ref{mainstructure}.

Given a finite or infinite $h$-uniform hypergraph $G=(V,E)$, we say that an $h$-tuple $(V_1,\ldots,V_h)$ of disjoint
subsets of $V$ is  {\em homogeneous} with respect to $G$ if either
all elements of $V_1 \times \ldots \times V_h$ are hyperedges of $G$
or none of them are. If, in addition, we have a finite non-atomic measure $\mu$ on $V$, we say that $G$ is {\em $(c,\mu)$-structured} provided that for all
disjoint measurable subsets $S_1,\ldots,S_h \subseteq V$, there
exist measurable $Y_i \subseteq S_i$ for $1 \leq i \leq h$ with
$\mu(Y_i) \geq c\mu(S_i)$ such that the $h$-tuple $(Y_1,\ldots,Y_h)$ is homogeneous. The following theorem generalizes a result in \cite{PaSo}, where the case $h=2$ was settled. The proof is based on an idea of Koml\'os.

\begin{theorem}\label{thm:regularity}
For any $c,\epsilon>0$ and for any positive integer $h$, there is
$K=K(h,c,\epsilon)$ such that the following statement is true. If
$G=(V,E)$ is a $(c,\mu)$-structured $h$-uniform hypergraph, where
$\mu$ is a non-atomic measure $\mu$, then for each $k \geq K$
there is a partition $V=V_1 \cup \ldots \cup V_k$ of $V$ into $k$
parts of equal measure such that all but at most an
$\epsilon$-fraction of the $h$-tuples $(V_{i_1},\ldots,V_{i_h})$ are
homogeneous.
\end{theorem}

In the sequel, let $G=(V,E)$ be a fixed $h$-uniform hypergraph with a finite non-atomic measure $\mu$ on $V$.
The {\it measure} of $G$, denoted by $\mu(G)$, is the product measure $\mu^h$ of the set of $h$-tuples $(v_1,\ldots,v_h) \in V^h$ with $\{v_1,\ldots,v_h\}$ an edge of $G$ (we assume throughout that this set is $\mu^h$-measurable). Define the {\it edge-density} $d(G)$ of $G$ to be $\frac{\mu(G)}{\mu(V)^h}$. For measurable vertex subsets $V_1,\ldots,V_h$, define
\begin{equation}\label{eq:def nu}
\nu(V_1,\ldots,V_h)\eqdef\mu^h\Big((v_1,\ldots,v_h)\in V_1 \times \cdots \times V_h:\  \{v_1,\ldots,v_h\}\in E\Big),
\end{equation}
and the {\it edge density} $$d(V_1,\ldots,V_h)\eqdef\frac{\nu(V_1, \ldots,V_h)}{\mu(V_1) \cdots \mu(V_h)}.$$

An $h$-tuple of disjoint subsets $(X_1,\ldots,X_h)$ of vertices in an $h$-uniform hypergraph is said to be {\it $(\gamma,\delta)$-superregular} if for any subsets $Y_i \subseteq X_i$ with $\mu(Y_1)\cdots\mu(Y_h) \geq \gamma \mu(X_1)\cdots \mu(X_h)$, we have $d(Y_1,\ldots,Y_h)\geq \delta$.

The following lemma shows that the vertex set of any dense $h$-uniform hypergraph can be partitioned into $h$ parts of equal measure such that the sub-hypergraph formed by all of its edges that contain one point from each part is still relatively dense. The analogue of this statement for finite hypergraphs is well known and very easy to prove, as a uniformly random partition of the vertex set into equal parts will almost surely work. Here we do not have the leisure of taking such a uniform random partition.

\begin{lemma}\label{partitesubhyp}
Let $G=(V,E)$ be an $h$-uniform hypergraph, $\mu$ a finite non-atomic measure on $V$. Then there is a partition $V=V_1 \cup \ldots \cup V_h$ into measurable subsets, each  of measure $\mu(V)/h$, such that $d(V_1,\ldots,V_h) \geq \frac{d(G)}{2}$.
\end{lemma}
\begin{proof}
Write $t=h\lceil h/d(G) \rceil$. Arbitrarily partition $V$ into $t$ subsets $V=U_1 \cup \ldots \cup U_t$, each of measure $\mu(V)/t$. This can be done since $\mu$ is non-atomic. The product measure of the set of $h$-tuples of vertices from distinct $U_i$ is at least $$\prod_{i=0}^{h-1}\left(1-\frac{i}{t}\right)\mu(V) \geq \left(1-\frac{{h \choose 2}}{t}\right)\mu(V)^h\geq \left(1-\frac{h^2}{2t}\right)\mu(V)^h \geq \left(1-\frac{d(G)}{2}\right)\mu(V)^h.$$ Hence, the edge density of the sub-hypergraph of $G$ with vertices in different $U_i$ is at least $d(G)-d(G)/2 = d(G)/2$. We randomly partition $V=V_1 \cup \ldots \cup V_h$, where each $V_i$ is a union of $t/h$ of the $U_j$, with each such partition being equally likely. For $U_{j_1},\ldots,U_{j_h}$ with $j_1<\ldots<j_h$, the probability that $U_{j_i} \subseteq V_i$ for all $1 \leq i \leq h$ is at least $h^{-h}$. Hence, by linearity of expectation and the fact $\mu(V_1)\cdots\mu(V_h)=h^{-h}\mu(V)^h$, the expected value of $d(V_1,\ldots,V_h)$ is at least $d(G)/2$. It follows that there is a partition $V=V_1 \cup \ldots \cup V_h$ into parts of equal measure with $d(V_1,\ldots,V_h) \geq d(G)/2$.
\end{proof}

The next lemma shows that if an $h$-tuple of sets is not $(\gamma,\delta)$-superregular,
then we can find subsets of large measure such that the edge density between them is significantly larger than the edge density between the original $h$-tuple.

\begin{lemma}\label{nextpartitesup}
Let $G=(V,E)$ be an $h$-uniform hypergraph and $\mu$ a finite non-atomic measure on $V$. If a collection of $h$ measurable vertex subsets $(W_1,\ldots,W_h)$ with $d(W_1,\ldots,W_h) = c$ is not $(\delta,\gamma)$-superregular, then there are subsets $Z_i \subseteq W_i$ for $1 \leq i \leq h$ such that
$$\mu(Z_1)\ldots\mu(Z_h) \geq \frac{\delta\gamma}{2^h} \mu(W_1)\cdots\mu(W_h)$$ and $$d(Z_1,\ldots,Z_h) \geq c+(c-2\delta)\frac{\gamma}{1-\gamma} .$$
\end{lemma}
\begin{proof}
Since $(W_1,\ldots,W_h)$ is not $(\delta,\gamma)$-superregular, there exist subsets  $Y_i \subseteq W_i$ for every $1 \leq i \leq h$ with
\begin{equation}\label{eq:contradict regular}
\mu(Y_1)\cdots\mu(Y_h) \geq \gamma\mu(W_1) \cdots\mu(W_h)\quad\mathrm{and}\quad d(Y_1,\ldots,Y_h) < \delta.\end{equation} 
The sum of all the $2^h$ terms $\nu(T_1,\ldots,T_h)$ (where $\nu$ is defined in~\eqref{eq:def nu}) with $T_i=Y_i$ or $T_i=W_i\setminus Y_i$ equals $\nu(W_1,\ldots,W_h)=c\mu(W_1)\cdots\mu(W_h)$. The sum of $\nu(T_1,\ldots,T_h)$ over all such terms with not all $T_i=Y_i$ and with $$\mu(T_1)\cdots\mu(T_h) \geq \frac{\delta\gamma}{2^h} \mu(W_1) \cdots \mu(W_h)$$ is therefore greater than $$\left(c-2^h\frac{\delta\gamma}{2^h}\right)\mu(W_1)\cdots\mu(W_h)-\delta\mu(Y_1)\cdots\mu(Y_h) \geq c\mu(W_1)\cdots\mu(W_h)-2\delta\mu(Y_1)\cdots\mu(Y_h).$$ Here we used both inequalities in~\eqref{eq:contradict regular}. Also, the sum of $\mu(T_1)\cdots \mu(T_h)$ over these terms is at most $\mu(W_1)\cdots\mu(W_h)-\mu(Y_1)\cdots\mu(Y_h)$. By averaging, if $a_1+\cdots+a_k \geq A$ and $b_1 +\cdots+b_k \leq B$ with all $a_i,b_i$ positive, then there is $1\le i\le k$ such that
$\frac{a_i}{b_i} \geq \frac{A}{B}$. Hence, there are $T_1,\ldots,T_h$ with $T_i=Y_i$ or $T_i=W_i \setminus Y_i$ and not all $T_i=Y_i$, with
$\mu(T_1)\cdots\mu(T_h) \geq 2^{-h}\delta\gamma \mu(W_1)\cdots\mu(W_h)$ and \begin{multline*} d(T_1,\ldots,T_h)  =  \frac{\nu(T_1,\ldots,T_h)}{\mu(T_1)\cdots\mu(T_h)} \geq  \frac{c\mu(W_1)\cdots\mu(W_h)-2\delta\mu(Y_1)\cdots\mu(Y_h)}{\mu(W_1)\cdots\mu(W_h)-\mu(Y_1)\cdots\mu(Y_h)} \\  =    c+(c-2\delta)\frac{\mu(Y_1)\cdots\mu(Y_h)}{\mu(W_1)\cdots\mu(W_h)-\mu(Y_1)\cdots\mu(Y_h)} \geq c+(c-2\delta)\frac{\gamma}{1-\gamma},
\end{multline*}
as required.
\end{proof}

By repeated application of Lemma~\ref{nextpartitesup}, we obtain the following result, which shows that a dense hypergraph contains a superregular $h$-tuple of sets of large measure.

\begin{lemma}\label{superreg}
For $\gamma,\delta>0$ and a positive integer $h$, there is $\alpha=\alpha(\gamma,\delta,h)$ such that the following holds.
If $G=(V,E)$ is an $h$-uniform hypergraph, $\mu$ is a finite non-atomic measure on $V$ with $d(G) \geq 8\delta$, then there is an $h$-tuple $(X_1,\ldots,X_h)$ of disjoint measurable vertex subsets, which is $(\gamma,\delta)$-superregular and satisfies $\mu(X_1)\cdots \mu(X_h) \geq \alpha \mu(V)^h$.
\end{lemma}
\begin{proof}
Define
$$
\alpha=\frac1{h^{h}}\left(\frac{\delta\gamma}{2^h}\right)^{\frac{2}{\gamma}\log_2 (1/\delta)}.
 $$
 By Lemma~\ref{partitesubhyp}, there is a partition $V=V_1 \cup \ldots \cup V_h$ into parts of equal measure such that $d\eqdef d(V_1,\ldots,V_h) \geq d(G)/2$. We will repeatedly apply Lemma~\ref{nextpartitesup}, starting with the sets $V_1,\ldots,V_h$, until we get a $(\gamma,\delta)$-superregular $h$-tuple.

If $(V_1,\ldots,V_h)$ is not $(\gamma,\delta)$-superregular, then we can find $V_1^1 \subseteq V_1,\ldots,V_h^1 \subseteq V_h$ with
$$\mu(V_1^1)\cdots\mu(V_h^1) \geq \frac{\delta\gamma}{2^h} \mu(V_1)\cdots\mu(V_h)$$ and
$$d\left(V_1^1,\ldots,V_h^1\right) \geq d+(d-2\delta)\frac{\gamma}{1-\gamma}  > d\left(1+\frac{\gamma}{2}\right).$$
After $k$ iterations, we either have found a $(\gamma,\delta)$-superregular $h$-tuple, or we find $V_1^k,\ldots,V_h^k$ with
$$\mu(V_1^k)\cdots\mu(V_h^k) \geq \left(\frac{\delta\gamma}{2^h}\right)^k \mu(V_1)\cdots\mu(V_h)$$ and $$d\left(V_1^k,\ldots,V_h^k\right) >
d\left(1+\frac{\gamma}{2}\right)^k.$$ This cannot continue for more than $k_0=\frac{2}{\gamma}\log_2(1/d)$ iterations, as otherwise we would produce an $h$-tuple of density more than $1$, a contradiction. Thus, at some step $k \leq k_0$, we find an $h$-tuple of sets $X_i \eqdef V_i^k$ with
\begin{equation*}\mu(X_1)\cdots\mu(X_h)  \geq  \left(\frac{\delta\gamma}{2^h}\right)^k \mu(V_1)\cdots\mu(V_h)=\frac{1}{h^h}\left(\frac{\delta\gamma}{2^h}\right)^k\mu(V)^h \geq \alpha \mu(V)^h,\end{equation*}
which is $(\gamma,\delta)$-superregular.
\end{proof}

An $h$-uniform hypergraph $H=(V,E)$ is {\em $h$-partite} if there is a partition $V=V_1 \cup \ldots \cup V_h$ such that every edge has exactly one vertex in each $V_i$. For a vertex set $V$, and a collection $\mathcal{C}$ of $h$-tuples $A_1,\ldots,A_h$ of vertex subsets of $V$, define the hypergraph $H(\mathcal{C})$ on $V$,
which is the union of the complete $h$-partite $h$-uniform hypergraphs with parts $A_1,\ldots,A_h$.

\begin{lemma}\label{regularity}
For $c,\epsilon>0$ and a positive integer $h$, there is $L=L(h,c,\epsilon)$ such that the following statement is true. If $G=(V,E)$ is a $(c,\mu)$-structured $h$-uniform hypergraph with $\mu$ a finite non-atomic measure, then there is a collection $\mathcal{C}$ of at most $L$ homogeneous $h$-tuples of vertex subsets such that the density of $H(\mathcal{C})$ is at least $1-\epsilon$.
\end{lemma}
\begin{proof}
Let $H_0$ denote the complete $h$-uniform hypergraph on $V$, and let $\gamma=c^h$, $\delta=\epsilon/8$, $\alpha=\alpha(\gamma,\delta,h)$ as in Lemma \ref{superreg}, $\beta=h!\delta c^h \alpha$, and $L=\beta^{-1}$.  Suppose that for some $i \geq 1$, we have already defined $H_{i-1}$. If $G$ has an $h$-tuple $A_1,\ldots,A_h$ of disjoint subsets which is homogeneous and such that the hypergraph which consists of those edges of $H_{i-1}$ that have one vertex in each $A_{\ell}$ has edge density at least $\beta$, then let $H_i$ denote the sub-hypergraph of $H_{i-1}$ obtained by deleting all edges with one vertex in each $A_{\ell}$.
Otherwise, we stop. This process will clearly terminate in $j \leq \beta^{-1}$ steps with an $h$-uniform hypergraph $H_j$ as $d(H_{i}) \leq 1-i\beta$ for each $i$.

We next show that $d(H_j) < \epsilon$. Indeed, otherwise by Lemma~\ref{superreg}, $H_j$ has a collection of $h$ disjoint vertex subsets $(X_1,\ldots,X_h)$ which is $(\gamma,\delta)$-superregular with $$\mu(X_1)\cdots\mu(X_h) \geq \alpha\mu(V)^h.$$ Since $G$ is $(c,\mu)$-structured, there are subsets $A_i \subseteq X_i$ for $1 \leq i \leq h$ with $\mu(A_i) \geq c \mu(X_i)$ such that $(A_1,\ldots,A_h)$ is homogeneous with respect to $G$. As $\gamma=c^h$,
$$\mu(A_1) \cdots \mu(A_h) \geq \gamma \mu(X_1) \cdots \mu(X_h),$$
and $(X_1,\ldots,X_h)$ is $(\gamma,\delta)$-superregular, we have $d(A_1,\ldots,A_h) \geq \delta$, and hence the hypergraph which consists of those edges of $H_{j}$ that have one vertex in each $A_{\ell}$ has edge density at least
$$h!\delta \cdot \frac{\mu(A_1)\cdots\mu(A_h)}{\mu(V)^{h}} \geq h!\delta c^h\cdot\frac{ \mu(X_1) \cdots \mu(X_h)}{\mu(V)^{h}} \geq h!\delta c^h \alpha=\beta.$$
However, this contradicts the fact that since the construction terminated at the $j$th step, there are no such subsets $A_1,\ldots,A_h$.
\end{proof}

For an $h$-uniform hypergraph $G=(V,E)$ and a vertex partition $\mathcal{P}=\{V = V_1 \cup \ldots \cup V_k\}$, the {\it homogeneous hypergraph} $G_{\mathcal{P}}$ is the $h$-uniform hypergraph on $V$ where $(v_1,\ldots,v_h) \in V_{i_1} \times \cdots \times V_{i_h}$ is an edge if and only if $(V_{i_1},\ldots,V_{i_h})$ is homogeneous with respect to $G$. Note that if $\mathcal{P}'$ is a refinement of $\mathcal{P}$, then $d(G_{\mathcal{P}'}) \geq d(G_{\mathcal{P}})$.
Given a collection $\mathcal{C}$ of $h$-tuples $(A_{1},\ldots,A_{h})$ which are homogeneous in $G$, define the partition $\mathcal{P}$ of $V$ into $(h+1)^{|\mathcal{C}|}$ parts, where each part consists of those vertices in the same $A_i$ (or in none of the $A_i$) for each $h$-tuple in $\mathcal{C}$. The hypergraph $G_{\mathcal{P}}$ contains the hypergraph $H(\mathcal{C})$. Hence, we have the following corollary of Lemma \ref{regularity} with $M=(h+1)^L$.

\begin{corollary}\label{regul1}
For any $c,\epsilon>0$ and for any positive integer $h$, there is $M=M(h,c,\epsilon)$ such that the following statement is true. If $G=(V,E)$ is a $(c,\mu)$-structured $h$-uniform hypergraph with a finite non-atomic measure $\mu$ on its vertex set, then there is a partition $\mathcal{P}$ of $V$ into at most $M$ parts such that $d(G_{\mathcal{P}}) \geq 1-\epsilon$.
\end{corollary}


\begin{proof}[Proof of Theorem~\ref{thm:regularity}]
Let $M=M(h,c,\frac{\epsilon}{2})$ as in Corollary \ref{regul1} and $K$ be the smallest integer which is at least $2Mh\epsilon^{-1}$. Fix $k\ge K$. By Corollary \ref{regul1}, there is a partition $\mathcal{P}'$ of $V$ into at most $M$ parts such that $d(G_{\mathcal{P}'}) \geq 1-\frac{\epsilon}{2}$. If $V_i$ is a part of $\mathcal{P}'$, arbitrarily partition $V_i$ into parts of measure $\mu(V)/k$ and one remaining piece of measure at most $\mu(V)/k$. Let $W$ be the union of the remaining parts, so $\mu(W) \leq M\mu(V)/k$, and arbitrarily partition $W$ into parts of measure $\mu(V)/k$. We have thus produced a partition $\mathcal{P}$ into $k$ parts of equal measure, and we next show that this partition satisfies the assertion of Theorem~\ref{thm:regularity}. The edge density of the hypergraph of $h$-tuples in $V^h$ that contain a vertex in $W$ is at most $$1-\left(1-\frac{\mu(W)}{\mu(V)}\right)^h \leq h\frac{\mu(W)}{\mu(V)} \leq \frac{hM}{k}\le\frac{\epsilon}{2}.$$ Thus $d(G_{\mathcal{P}}) \geq
d(G_{\mathcal{P}'})-\frac{\epsilon}{2} \geq 1-\epsilon$, and hence the partition $\mathcal{P}$ satisfies the conclusion of Theorem~\ref{thm:regularity}.
\end{proof}

\section{A partition result---Proof of Theorem~\ref{mainstructure}}\label{structuresub}

Before proving Theorem \ref{mainstructure} in its full generality, we give a simple argument for the special case $d=2$, which provides a good upper bound on the constant $K(\epsilon,2)$.
\begin{proposition}
Let $\epsilon>0$ and $k \geq \frac{12}{\epsilon}+1$. For any finite Borel measure $\mu$ on $\mathbb{R}^2$ with respect to which every line has measure $0$, and for any point $q \in \mathbb{R}^d$, there is a partition $\mathbb{R}^2=A_1 \cup \ldots \cup A_{k}$ into $k$ measurable parts of equal measure, such that all but at most an $\epsilon$-fraction of the triples $A_{h},A_{i},A_{j}$ are homogenous with respect to $q$. \end{proposition}
\begin{proof}
Partition $\mathbb{R}^2$ radially around $q$ into $k$ cones $A_1, \ldots, A_{k}$ of equal measure. Notice that a triple $A_h,A_i,A_j$ is not homogeneous with respect to $q$ if and only if one of them intersects the reflection of another about $q$. Since the number of such triples of cones is at most $2k(k-2)$, the fraction of nonhomogeneous triples cannot exceed $2k(k-2)/{k\choose 3}=12/(k-1)\leq \epsilon$, which completes the proof.
\end{proof}

Next we turn to the proof of Theorem \ref{mainstructure} in the general case. We break the proof into four lemmas. For the first one, we recall Radon's theorem, which states that any set of $d + 2$ points in $\mathbb{R}^d$ can be partitioned into two sets whose convex hulls have nonempty intersection (see \cite{Ec}).

\begin{lemma}\label{lemradon}
Let $v_1,\ldots,v_{d+1} \in \mathbb{R}^d$. A point $q \in \mathbb{R}^d$ belongs to the simplex with vertex set $V=\{v_1,\ldots,v_{d+1}\} \subseteq \mathbb{R}^d$ if and only if for each nonempty proper subset $X \subseteq V$, there is a hyperplane passing through $q$ which separates $X$ from $V \setminus X$.
\end{lemma}

\begin{proof}
In one direction the statement is clear, as there is a hyperplane through any internal point separating any proper subset of the vertex set from its complement.
In the other direction, suppose $q$ is not in the simplex, and consider the set $V \cup \{q\}$ of $d+2$ points. By Radon's theorem, there is a partition $V \cup \{q\} = A \cup B$ such that the convex hull of $A$ and the convex hull of $B$ have a point $p$ in common. Suppose without loss of generality that $q \in A$. Clearly, $|A|>1$, because $q$ is outside of the simplex. But then $A\setminus \{q\}$ and $B$ cannot be separated by a hyperplane passing through $q$.
\end{proof}

We need the following version of the ham sandwich theorem (see, e.g., \cite{Ma}).

\begin{lemma}\label{hamsandwich}
Let $S_1,\ldots,S_{d-1}$ be measurable subsets of $\mathbb{R}^d$ and $q \in \mathbb{R}^d$. There is a hyperplane through $q$ that partitions each $S_i$ into two parts of equal measure.
\end{lemma}

By repeated application of Lemma~\ref{hamsandwich}, and then using Lemma~\ref{lemradon}, we obtain:

\begin{lemma}\label{imptlemma}
For any positive integer $d$, there is $c_d>0$ satisfying the following condition. Let $\mu$ be a finite Borel measure on $\mathbb{R}^d$ with respect to which every hyperplane has measure $0$, let $S_1,\ldots,S_{d+1}$ be measurable subsets of $\mathbb{R}^d$, and let $q \in \mathbb{R}^d$. Then there exist $Y_i \subseteq S_i$ for all $1 \leq i \leq d+1$ such that $\mu(Y_i) \geq c_d\mu(S_i)$ and $Y_1,\ldots,Y_{d+1}$ are homogeneous with respect to $q$.
\end{lemma}
\begin{proof}
Consider an arbitrary labeling $X_1,\ldots,X_{2^{d+1}-2}$ of the nonempty proper subsets of $[d+1]=\{1,\ldots,d+1\}$. We describe an iterative process for constructing the desired sets $Y_1,\ldots,Y_{d+1}$.

Let $S^0_j=S_j$, for every $j \in [d+1]$. After completing step $i$, we will have subsets $S^i_j \subseteq S_j$ with $\mu(S^i_j) \geq 2^{-i}\mu(S_j)$ such that at least one of the following two conditions is satisfied:
\begin{enumerate}
\item there is a hyperplane through $q$ such that $S^i_1 \cup S^i_2 \cup  \ldots \cup S^i_{d+1}$ lies entirely on one of its sides, or
\item for every $1 \leq k \leq i$, there is a hyperplane through $q$ that separates the sets $\{S^i_j\}_{j \in X_k}$ from the sets $\{S^i_j\}_{j \in [d+1] \setminus X_k}$.
\end{enumerate}
The proof shows that we can take $c_d=2^{2-2^{d+1}}$. Notice that the inductive hypothesis holds vacuously at the end of step $0$. Suppose we have already completed step $i$ and we wish to proceed to the next step.

If condition (1) is satisfied, then we simply let $S_j^{i+1}= S^i_j$.
Thus, we may suppose that condition (2) is satisfied. Let $a \in X_{i+1}$ and $b \in [d+1] \setminus X_{i+1}$. We apply Lemma \ref{hamsandwich} to the $d-1$ sets $S^i_j$ with $j \in [d+1] \setminus \{a,b\}$. There is a hyperplane $H$ containing $q$ that separates each such $S^i_j$ into parts of equal measure.
The hyperplane $H$ partitions $S^i_a$ into two subsets. Let $S^{i+1}_a$ be the subset of larger measure. Similarly, $H$ partitions $S^i_b$ into two subsets.
Let $S^{i+1}_b$ be the subset of larger measure. If $S^{i+1}_a$ and $S^{i+1}_b$ are on the same side of $H$, then let $S^{i+1}_j$ for
$j \in [d+1] \setminus \{a,b\}$ be the subset of $S^i_j$ consisting of those points on the same side of $H$ as $S^{i+1}_a$ and $S^{i+1}_b$. In this case, we have
$\mu(S^{i+1}_j) \geq \frac{1}{2}\mu(S^i_j) \geq 2^{-i-1}\mu(S_j)$, the first of the two desired properties holds, and we have completed step $i+1$. Otherwise,
for $j \in X_{i+1} \setminus \{a\}$, let $S^{i+1}_j$ be the subset of $S^i_j$ consisting of those points on the same side of $H$ as $S^{i+1}_a$, and
for $j \in [d+1] \setminus (X_{i+1} \cup \{b\})$, let $S^{i+1}_j$ be the subset of $S^i_j$ consisting of those points on the same side of $H$ as $S^{i+1}_b$.
By construction, we have $\mu(S^{i+1}_j) \geq \frac{1}{2}\mu(S^i_j) \geq 2^{-i-1}\mu(S_j)$,
the second of the two desired properties holds, and we have completed step $i+1$.

We may therefore assume that we finish the iterative process, and in the end we have sets $Y_j=S^{2^{d+1}-2}_j$ with
$\mu(Y_j) \geq 2^{2-2^{d+1}}\mu(S_j)$ for $1 \leq j \leq d+1$, and
for each $1 \leq k \leq 2^{d+1}-2$, there is a hyperplane through $q$ that separates the sets $\{Y_j\}_{j \in X_k}$
from the sets $\{Y_j\}_{j \in [d+1] \setminus X_k}$. By Lemma \ref{lemradon}, this implies that every simplex with one vertex in each $Y_j$ contains $q$.
\end{proof}

We are now ready to complete the proof of Theorem \ref{mainstructure}. Given a point $q \in \mathbb{R}^d$, define the hypergraph $H_q$ with vertex set $\mathbb{R}^d$ as the set of all $(d+1)$-tuples of points whose convex hulls contain $q$. Lemma \ref{imptlemma} states that for each finite measure $\mu$ on $\mathbb{R}^d$ such that every hyperplane has measure $0$, the hypergraph $H_q$ is $(c_d,\mu)$-structured. Theorem \ref{mainstructure} then follows from Theorem~\ref{thm:regularity}.

\section{Optimal sparse constructions in space---Proof of Theorem~\ref{usestructure}}\label{optimal}

In this section, we deduce Theorem~\ref{usestructure} from Corollary~\ref{discretestructure}.
Let $H=(V,E)$ be a $(d+1)$-uniform hypergraph. The {\it edge density} $$\rho(H)\stackrel{\mathrm{def}}{=}\frac{|E|}{{|V| \choose d+1}}$$ is the fraction of $(d+1)$-tuples that are hyperedges of $H$. For vertex subsets $V_1,\ldots,V_{d+1} \subseteq V$, the edge density $\rho(V_1,\ldots,V_{d+1})$ is defined as the fraction of $(d+1)$-tuples in $V_1 \times \ldots \times V_{d+1}$ that are hyperedges of $H$.

\medskip

\noindent{\bf Proof of Theorem~\ref{usestructure}}.
Let $t=8d^2/\epsilon$, so that for any equipartition of the vertex set of a complete $(d+1)$-uniform
hypergraph  into $k \geq t$ parts, all but at most an $\frac{\epsilon}{8}$-fraction of its hyperedges have their vertices in different parts.
Indeed, the fraction of hyperedges with one vertex in each part is $$\left(\frac{n}{k}\right)^{d+1}\cdot\frac{{k \choose d+1}}{{n \choose d+1}} \geq \prod_{i=1}^{d} \left(1-\frac{i}{k}\right) \geq 1-\frac{d^2}{k}
\geq 1-\frac{\epsilon}{8}.$$ Let $K=K(\epsilon/8,d)$ be the constant from Corollary \ref{discretestructure}, and $k=\max\{K,t\}$.

Let $r(d,\epsilon)$ be sufficiently large so that for any $r \geq r(d,\epsilon)$ and $n$ a sufficiently large multiple of $d+1$,
there is an $r$-regular $(d+1)$-uniform hypergraph $H$ on $n$ vertices, whose hyperedges are uniformly distributed in the sense that for any disjoint vertex subsets $V_1,\ldots,V_{d+1} \subseteq V(H)$ with $|V_i|\geq \frac{n}{k}$ for $1 \leq i \leq d+1$,
\begin{eqnarray}\label{densityineq}\left|\frac{\rho(V_1,\ldots,V_{d+1})}{\rho(H)}-1\right| \leq \frac{\epsilon}{4}.\end{eqnarray}
The existence of an integer $r(d,\epsilon)$ and a hypergraph $H$ with the above properties follows from the standard fact that an $n$-vertex $r$-regular $(d+1)$-uniform hypergraph $H$ chosen uniformly at random from all $n$-vertex $r$-regular $(d+1)$-uniform hypergraphs, meets the requirements for large enough $r$ with probability tending to $1$  as $n\to \infty$.

Let $P \subseteq \mathbb{R}^d$ satisfy $|P|=n$. Since $n$ is sufficiently large, there is a point $q$ that is in at least $c(d)-\frac{\epsilon}{4}$ fraction of the simplices with vertices in $P$. Since $k \geq K$, by Corollary \ref{discretestructure}, there is an equipartition $P=P_1 \cup \ldots \cup P_k$ such that all but at most an $\frac{\epsilon}{8}$-fraction of the $(d+1)$-tuples $P_{i_1},\ldots,P_{i_{d+1}}$ are homogenous with respect to $q$. Since $k \geq t$, all but at most an $\frac{\epsilon}{8}+\frac{\epsilon}{8}=\frac{\epsilon}{4}$-fraction of the $(d+1)$-tuples of points of $P$ have their vertices in $d+1$ different parts of the partition, and these parts are homogeneous.
Since $q$ is in at least a fraction $c(d)-\frac{\epsilon}{4}$ of the simplices with vertices in $P$, at least a fraction $c(d)-\frac{\epsilon}{2}$ of the $(d+1)$-tuples of points of $P$ span a simplex containing $q$ and having its vertices in $d+1$ different parts of the partition such that these parts are homogeneous.

Let $f:V(H) \rightarrow P$ be an arbitrary bijection between the vertices of $H$ and $P$. Write $V_i=f^{-1}(P_i)$ for $1 \leq i \leq k$. Note that if $P_{i_1},P_{i_2},\ldots,P_{i_{d+1}}$ are homogeneous with respect to $q$ such that there is a simplex containing $q$ with one vertex in each of these parts, then necessarily all the simplices with one vertex in each of these parts contains $q$. By (\ref{densityineq}), for all $d+1$ parts $V_{i_1}, \ldots, V_{i_{d+1}}$, the hyperedge density in $H$ between these parts is at least $\left(1-\frac{\epsilon}{4}\right)\rho(H)$. Putting this together with the previous paragraph, we get that at least a fraction $\left(c(d)-\frac{\epsilon}{2}\right)\left(1-\frac{\epsilon}{4}\right) \geq c(d)-\epsilon$ of the hyperedges of $H$ induce simplices containing $q$. We conclude that $c(H) \geq c(d)-\epsilon$.
\qed

\section{Sparse hypergraphs with nearly maximal overlap number}\label{tightsubsection}

Here we prove the following result, which implies Theorem \ref{tightness}.

\begin{theorem} \label{tightconstruction}
Let $d$ and $\Delta$ be positive integers and $\epsilon>0$. If $n \geq 2^9\epsilon^{-3}d^{9}\Delta^3$ and
$H$ is a $(d+1)$-uniform hypergraph with $n$ vertices, maximum degree $\Delta$, and without isolated vertices, and $P$ is a set of $n$ points in $\mathbb{R}^d$
such that no point in $\mathbb{R}^d$ is in a fraction more than $c$ of the simplices with vertices in $P$,
then there is a bijection $f:V(H) \to P$ such that no point of $\mathbb{R}^d$
is in a fraction more than $c+\epsilon$ of the simplices whose vertices are the image by $f$ of a hyperedge of $H$.
\end{theorem}

\noindent We actually show that almost surely we may take $f$ to be a uniform random bijection.

We shall use below Azuma's inequality (see, e.g., \cite{AlSp}), which asserts that if $c=X_0,\ldots,X_n$ is a martingale with $|X_{a+1}-X_a| \leq t$ for all $0 \leq a \leq n-1$, then
\begin{equation}\label{eq:azuma}\Pr\Big[|X_n-c|>\lambda\sqrt{n}\Big]<2e^{-\frac{\lambda^2}{2t^2}}.
\end{equation}


Let $H$ and $F$ be hypergraphs each with the same number of vertices. For a bijection $f:V(H) \rightarrow V(F)$, let $m_f$ denote the number of hyperedges of
$H$ whose image is a hyperedge of $F$.

\begin{lemma}\label{applyAzuma}
Let $H$ and $F$ be $k$-uniform hypergraphs each with $n$ vertices such that $H$ has maximum degree $\Delta$.
Then the probability that for a random bijection $f:V(H) \rightarrow V(F)$, the number $m_f$ deviates from $|E(H)|\cdot|E(F)|/{n \choose k}$ by more than
$\lambda \sqrt{n}$ is at most $2e^{-\frac{\lambda^2}{2(2k+1)^2\Delta^2}}$.
\end{lemma}
\begin{proof}
Let $e_1,\ldots,e_{|E(H)|}$ denote the hyperedges of $H$. For a hyperedge $e_j$, let $X(e_j)$ be the indicator random variable of the event that the image of $e_j$ by $f$ is a hyperedge of $F$. That is, $X(e_j)=1$ if $f(e_j)$ is a hyperedge of $F$, and $X(e_j)=0$ otherwise. Let $Y$ denote the random variable counting the number of hyperedges of $H$ whose image is a hyperedge of $F$, so
$Y=\sum_{j=1}^{|E(H)|} X(e_j)$.  Each hyperedge of $H$ has a probability $|E(F)|/{n \choose k}$ of being mapped by $f$ to a hyperedge of $F$. By linearity of expectation, the expected value of $Y$ is $\mathbb{E}[Y]=|E(H)|\cdot|E(F)|/{n \choose k}$.

Let $V(H)=[n]=\{1,\ldots,n\}$. For $a=0,\ldots,n$, let $X_a(e_j)$ be the probability of the event that the image of $e_j$ by $f$ is a hyperedge of $F$ after picking $f(1),\ldots,f(a)$, and let $Y_a\stackrel{\mathrm{def}}{=}\sum_{j=1}^{|E(H)|}X_a(e_j)$ denote the expected value of $Y$ after picking $f(1),\ldots,f(a)$.
So $Y_0$ denotes the expected value of $m_f$, which is $|E(H)|\cdot|E(F)|/{n \choose k}$.

By construction, $\{Y_a\}_{a=0}^n$ is a martingale.
We next give an upper bound on $$|Y_{a+1}-Y_a| \leq \sum_{j=1}^{|E(H)|} |X_{a+1}(e_j)-X_a(e_j)|,$$ so as to apply Azuma's inequality~\eqref{eq:azuma}.
For the $\le\Delta$ hyperedges $e_j$ that contain $a+1$, we bound $|X_{a+1}(e_j)-X_a(e_j)| \leq 1$. For those hyperedges with all  vertices in $\{1,\ldots,a\}$, we have $X_{a+1}(e_j)=X_a(e_j)$. Let $e_j$ be a hyperedge which does not contain ${a+1}$, and contains a vertex which is more than $a+1$.
 Let $i_1,\ldots,i_h$ be the vertices of $e_j$ that are at most $a$, so $h<k$ as $e_j$ contains a vertex which is more than $a+1$. Let $w(e_j)=k-h$ denote the number of vertices of $e_j$ which are greater than $a+1$. All of these vertices are in $\{a+2,\ldots,n\}$, and therefore $n-a\ge k-h+1$. It follows that
\begin{multline}\label{eq:for denominator}
n-a-w(e_j)=n-a-k+h=\frac{n-a}{k}+\left(1-\frac{1}{k}\right)(n-a)-k+h\\\ge \frac{n-a}{k}+\left(1-\frac{1}{k}\right)(k-h+1)-k+h=\frac{n-a-1}{k}+\frac{h}{k} \geq \frac{n-a-1}{k}.
\end{multline}
 Let $Z$ denote the number of hyperedges of $F$ that contain $f(i_1),\ldots,f(i_h)$ and whose remaining vertices are in $V(F) \setminus \{f(1),\ldots,f(a)\}$, and $Z'$ denote the number of hyperedges of $F$ that contain $f(i_1),\ldots,f(i_h)$ and whose remaining vertices are in $V(F) \setminus \{f(1),\ldots,f(a),f(a+1)\}$. We have $Z \geq Z' \geq Z-{n-a-1 \choose k-h-1}$, as $f(i_1),\ldots,f(i_h),f(a+1)$ are in at most ${n-a-1 \choose k-h-1}$ hyperedges of $F$ whose remaining vertices are in $V(F) \setminus \{f(1),\ldots,f(a),f(a+1)\}$. We also have $Z \leq {n-a \choose k-h}$. Note that
 $X_{a}(e_j)=\frac{Z}{{n-a \choose k-h}}$ and $X_{a+1}(e_j)=\frac{Z'}{{n-a-1 \choose k-h}}.$
Hence,
\begin{eqnarray*}
|X_{a}(e_j)-X_{a+1}(e_j)| & = & \left|\frac{Z}{{n-a \choose k-h}}-\frac{Z'}{{n-a-1 \choose k-h}}\right| \\ & \leq &  Z\left({n-a-1 \choose k-h}^{-1}-{n-a \choose k-h}^{-1}\right)+(Z-Z'){n-a-1 \choose k-h}^{-1} \\ & \leq & \left(\frac{n-a}{n-a-k+h}-1\right)+\frac{k-h}{n-a-k+h} \\ & = & \frac{2(k-h)}{n-a-k+h} \\&=&2\frac{w(e_j)}{n-a-w(e_j)},\end{eqnarray*}
where the first inequality is the triangle inequality, the second inequality follows from substituting in $Z \leq {n-a \choose k-h}$ and $Z-Z' \leq {n-a-1 \choose k-h-1}$.

We have $\sum w(e_j) \leq \Delta(n-a-1)$, where the sum is over all hyperedges that contain a vertex greater than $a+1$,
as each vertex has degree at most $\Delta$. Hence, the sum of $|X_{a}(e_j)-X_{a+1}(e_j)|$ over all vertices that contain a vertex greater than $a+1$
is at most $$2\sum \frac{w(e_j)}{n-a-w(e_j)} \leq 2\Delta k,$$
where we used~\eqref{eq:for denominator}.
Putting this altogether, we have
$|Y_{a+1}-Y_{a}| \leq \Delta (2k+1)$. By Azuma's inequality~\eqref{eq:azuma}, the probability that $Y_n=Y=m_f$ differs from $X_0=|E(H)|\cdot|E(F)|/{n \choose k}$ by more than $\lambda\sqrt{n}$ is at most $2e^{-\frac{\lambda^2}{2(2k+1)^2\Delta^2}}$.
\end{proof}

We use the following well known fact (see \cite{Po}, pp. 43--52).

\begin{lemma} \label{regions}
Any set $P$ of $n$ points in $\mathbb{R}^d$ determine at most $n^{d^2}$ regions, such that any two points in the same region are in the same collection of
simplices with vertices in $P$.
\end{lemma}

The two previous lemmas are all the tools we need to complete the proof of Theorem \ref{tightconstruction}.

\vspace{0.2cm}
\noindent{\bf Proof of Theorem \ref{tightconstruction}:} Since $H$ does not have isolated vertices, the number $|E(H)|$ of hyperedges of $H$ is at least
$n/(d+1)$. Let $P$ be a set of $n$ points in $\mathbb{R}^d$ such that no point $q \in \mathbb{R}^d$ is in more than a fraction $c$ of the simplices
whose vertices are in $P$. By Lemma \ref{regions}, $P$ determines at most $n^{d^2}$ regions, such that any two points in the same region are in the same collection of
simplices with vertices in $P$. Let $q$ be a representative point for a region, $c_q \leq c$ denote the fraction of simplices with vertices in $P$ that contain $q$, and let $F_q$ denote the hypergraph on $P$ consisting of all simplices with vertices in $P$ that contain the point $q$.

Let $f:V(H) \rightarrow P$ be a bijection taken uniformly at random. By Lemma \ref{applyAzuma}, the probability that the fraction of hyperedges of $H$ which map to simplices containing $q$ is at least $c_q+\epsilon$ is at most $2e^{-\frac{\lambda^2}{2(2d+3)^2\Delta^2}}$, where $\lambda\sqrt{n}=\epsilon|E(H)| \geq \epsilon n/(d+1)$.
However, there are at most $n^{d^2}$ such hypergraphs $F_q$, and since $n \geq 2^9\epsilon^{-3}d^{9}\Delta^3$, the probability is at most
\begin{equation*}n^{d^2}e^{-\frac{\lambda^2}{2(2d+3)^2\Delta^2}} \leq \exp\left(d^2\log n - \frac{(\epsilon\sqrt{n}/(d+1))^2}{2(2d+3)^2\Delta^2}\right)\le
\exp\left(d^2\log n -\frac{\e^2n}{2(2d+3)^4\Delta^2}\right)=o(1)\end{equation*} that there is a point in $\mathbb{R}^d$
contained in more than a fraction $c+\epsilon$ of the hyperedges of $H$. Thus, with high probability,
a uniformly random bijection  has the desired property. \qed

\section{Concluding remarks}\label{remarks}

An alternative proof of Theorem~\ref{29result} was given by Gromov~\cite{Gro2}, based on an application of Garland's vanishing theorem to finite quotients of certain Bruhat-Tits building. Gromov's argument fails for $d>2$, and does not yield the sharp constant $\frac29$ as in Theorem~\ref{29result}, yet his  construction has some remarkable stronger properties which we now describe.

Gromov investigated~\cite{Gro2} the role of the fact that the edges of the triangles in the Boros-F\"uredi theorem are assumed to be straight line segments. He showed that  it suffices to replace ``straight lines" by sufficiently regular Jordan arcs. Gromov's construction based on Garland's theorem enjoys this stronger property as well (see Section 2.10 in~\cite{Gro2}): one just needs the associated mapping from Gromov's simplicial complex to $\R^2$ to be continuous, the image of each edge to be nowhere dense in $\R^2$, and that its restriction to each face is at most $r$-to-$1$ (in which case the resulting bounds depend on $r$). At present it is unknown whether for $d\ge 3$ there exist arbitrarily large bounded degree $d$-dimensional simplicial complexes which are highly overlapping with respect to non-affine embeddings into $\R^d$.

An inspection of the construction of Section~\ref{sec:high dim} reveals that, if the graph $G$ has girth greater than $2d$, then the resulting bounded degree $d$-dimensional highly overlapping simplicial complexes admit a continuous and piecewise affine retraction onto their $1$-skeleton. It follows that for these complexes,  if we replace ``simplices" by ``generalized simplices" whose edges are allowed to be continuous and piecewise affine arcs rather than straight lines, then the conclusion that there must be a point in a constant fraction of these generalized simplices, corresponding to an embedding of their vertex set into $\R^d$, would fail. Thus, the situation in the sparse setting is subtle, and passing from the case of affine mapping to more general continuous mappings is not automatic.

It follows for instance from \cite{Ba} and \cite{Pa} that for any system of at least constant times $n^3$ triangles induced by a set of $n$ points in the plane in general position, there is a point covered by at least a positive fraction of all triangles. In the present paper, we studied sparse systems of triangles with similar properties. Another closely related question is the following. For any positive integers $n$ and  $1<t<{n\choose 3}$, determine the largest number $m$ such that for any system of at least $t$ triangles induced by a set of $n$ points in the plane, there is a point contained in at least $m$ triangles. See~\cite{ACE}, \cite{NiSh}. The best known general lower bound is roughly $t^3/n^6$, but for most values of the parameters this is probably a very weak bound.

\medskip
\vspace{0.2cm}
\noindent {\bf Acknowledgements.} We thank G\'abor Tardos for many useful conversations and ideas, and Yves Benoist for his help on Section~\ref{sec:building}.


\bibliographystyle{abbrv}

\bibliography{overlap}

\bigskip
\bigskip
\footnotesize{
\noindent {\bf Jacob Fox:} Department of Mathematics, Princeton University. Email: {\tt jacobfox@math.princeton.edu}. Research
supported by an NSF Graduate Research Fellowship and a Princeton
Centennial Fellowship.

\medskip

\noindent{\bf Mikhail Gromov:} IHES and Courant
Institute, NYU. Email: {\tt gromov@cims.nyu.edu}.

\medskip

\noindent{\bf Vincent Lafforgue:} CNRS, Universit\'e Paris 7 Denis Diderot. Email: {\tt vlafforg@math.jussieu.fr}.

\medskip

\noindent{\bf Assaf Naor:} Courant Institute, NYU. Email: {\tt naor@cims.nyu.edu}. Research supported  by NSF grants CCF-0635078 and CCF-0832795, BSF grant 2006009, and the Packard Foundation.

\medskip

\noindent{\bf J\'anos Pach:} EPFL, CUNY, and Courant Institute, NYU. Email: {\tt pach@cims.nyu.edu}. Supported by NSF Grant CCF-08-30272, and by grants from NSA, PSC-CUNY, the Hungarian Research Foundation OTKA,
and BSF.
}

\end{document}